\documentclass[a4paper,11pt]{article}
\usepackage{fullpage}
\usepackage{amsmath,amsfonts,amssymb,amsthm}
\usepackage{mathabx}
\usepackage{dsfont}
\usepackage{calrsfs}
\usepackage{graphicx}
\usepackage{mathrsfs}
\usepackage{color}
\usepackage[dvipsnames]{xcolor}
\usepackage{tikz}
\usetikzlibrary{arrows.meta}
\usepackage{hyperref}
\numberwithin{equation}{section}

\providecommand{\R}{\mathbb{R}}
\providecommand{\C}{\mathbb{C}}
\providecommand{\N}{\mathbb{N}}

\renewcommand{\Re}{\mathop{\text{Re}}}
\renewcommand{\Im}{\mathop{\text{Im}}}
\renewcommand{\div}{\operatorname{div}}
\newcommand{\curl}{\operatorname{curl}}
\newcommand{\Id}{\operatorname{Id}}
\newcommand{\supp}{\operatorname{Supp}}

\newcommand{\sq}{\mathchoice
    {\raisebox{-0.2ex}{\scalebox{1.7}{$\diamond$}}} 
    {\raisebox{-0.2ex}{\scalebox{1.7}{$\diamond$}}}
    {\raisebox{-0.1ex}{\scalebox{1.2}{$\diamond$}}}
    {\raisebox{-0.1ex}{\scalebox{0.9}{$\diamond$}}}  }
\newcommand{\sql}{\mathchoice
    {\raisebox{-0.2ex}{\scalebox{0.8}[1.7]{$\triangleleft$}}} 
    {\raisebox{-0.2ex}{\scalebox{0.8}[1.7]{$\triangleleft$}}}
    {\raisebox{-0.1ex}{\scalebox{0.4}[1.2]{$\triangleleft$}}}
    {\raisebox{-0.1ex}{\scalebox{0.4}[1]{$\triangleleft$}}}  }
\newcommand{\sqr}{\mathchoice
    {\raisebox{-0.2ex}{\scalebox{0.8}[1.7]{$\triangleright$}}} 
    {\raisebox{-0.2ex}{\scalebox{0.8}[1.7]{$\triangleright$}}}
    {\raisebox{-0.1ex}{\scalebox{0.4}[1.2]{$\triangleright$}}}
    {\raisebox{-0.1ex}{\scalebox{0.4}[1]{$\triangleright$}}}  }
\newcommand{\SN}{\mathcal{M}_{\uparrow}}
\newcommand{\NS}{\mathcal{M}_{\downarrow}}
\newcommand{\Hsq}{H_{\sq}}

\newcommand{\A}{\mathfrak{a}}
\newcommand{\Ba}{\mathcal{B}}

\newcommand{\II}{I{\mkern-2mu}I}

\newtheorem{remark}{Remark}

\newtheorem{theorem}{Theorem}
\newtheorem*{theorem*}{Theorem}
\newtheorem{manualtheoreminner}{Theorem}
\newenvironment{manualtheorem}[1]{%
  \renewcommand\themanualtheoreminner{#1}%
  \manualtheoreminner
}{\endmanualtheoreminner}
\newtheorem{lemma}{Lemma}[section]
\newtheorem{proposition}{Proposition}

\title{On the cost of reaching an attainable state \\ for the boundary-controlled $1$-D heat equation}
\author{Olivier Glass}
\begin{document}
\maketitle
\begin{abstract}
We consider the heat equation in an interval of the real line, with null initial condition,
and Dirichlet boundary controls. We study the cost of reaching an attainable state at some time $T>0$.
Precisely, given a state in the interval that we know is reachable by suitable boundary controls,
we estimate the size of the smallest controls that reach it, in terms of an appropriate norm of this state.
Our estimate is optimal in some sense as the time horizon goes to $0$.
\end{abstract} 
\ \par
\ \par
{\small
} \par
\ \par
%
%
%
%
%
\section{Introduction}
In this paper, we consider the classical problem of describing the reachable space 
for the boundary controlled heat equation in one space dimension.
The problem is as follows. Let $L>0$, $T>0$.
Consider the heat equation in the domain $[0,T] \times [-L,L]$, with boundary controls 
at $x=-L$ and $x=L$ and null initial condition:
\begin{alignat}{2} 
\label{Eq:Heat}
&\partial_t v - \partial^2_{xx} v = 0 \ &&\text{ in } \ (0,T) \times (-L,L), \\
\label{Eq:Heat-IC}
&v(0,x) = 0 \ &&\text{ for } \ x \in (-L,L), \\
\label{Eq:Heat-BC}
&v(t,-L) = g_-(t) , \ \  v(t,L) = g_+(t) 
\ \ \ && \text{ for } \ t \in (0,T).
\end{alignat}
Above, $g_-$ and $g_+$ are the boundary controls taken for instance in $L^2(0,T)$,
with values in $\C$, which simplifies the presentation.
The classical problem of reachability for system 
\eqref{Eq:Heat}--\eqref{Eq:Heat-BC}
refers to the description of the reachable space
\begin{equation} \label{Eq:ReachableSpace}
\mathcal{R}_T :=\left\{ v(T,\cdot), \ \ g_- , \ g_+ \in L^2(0,T) \right\}.
\end{equation}

The complete description of this space is now fully understood due to the 
works of Orsoni \cite{Orsoni}, Hartmann-Orsoni \cite{Hartmann-Orsoni} and 
Kellay-Normand-Tucsnak \cite{Kellay-Normand-Tucsnak}, following a work by 
Hartmann-Kellay-Tucsnak \cite{hartmann2020reachable}. 
It turns out that the space $\mathcal{R}_T$ can be described as restrictions to 
$(-L,L)$ of functions belonging to the Bergman space of the following square in the complex space:
\begin{equation*}
\sq = \left\{ z \in \C \ \Big/ \ |\Re(z)| + |\Im(z)| <L \right\}.
\end{equation*}
The Bergman space $L^2_a(\sq)$ is the space of all functions in $L^2(\sq)$ 
(for the area measure in $\C$), which are holomorphic in $\sq$. 
With this notation, the above references prove that
\begin{equation} \label{Eq:Tucsnak_et_al}
\mathcal{R}_T = L^2_a(\sq)_{|(-L,L)}.
\end{equation}
Note that this space does not depend on $T>0$, which corresponds to a classical result 
for systems which are null controllable for any time, 
see Fattorini \cite{fattorini1978reachable} and Seidman \cite{seidman1979time}. 
An elementary proof of this fact can also be found in 
\cite[Proposition 3.4]{Kellay-Normand-Tucsnak}. 
In addition, it can be noticed that the reachable states starting 
from some initial condition $u_0 \in L^2(-L,L)$ rather than $0$ in \eqref{Eq:Heat-IC} 
would be the same, due to the null-controllability of the system in any time, 
see Fattorini-Russell \cite{Fattorini-Russell}. \par
Now based on \eqref{Eq:Tucsnak_et_al}, it is then a direct consequence 
of the open mapping theorem that for each $T>0$, 
there exists a constant $C_T>0$ such that for any $v \in L^2_a(\sq)$, 
there exist controls $g_-$ and $g_+$ in $L^2(0,T)$ such that the corresponding solution 
of \eqref{Eq:Heat}--\eqref{Eq:Heat-BC} reaches $v_{|(-L,L)}$ at time $T>0$ 
and moreover satisfy
\begin{equation} \nonumber 
\|g_-\|_{L^2(0,T)} + \|g_+\|_{L^2(0,T)} \leq C_T \|v \|_{L^2(\sq)}.
\end{equation}

\ \par
The goal of this paper is to study this constant $C_T$, 
in particular as $T \rightarrow 0^+$.
To state our result, we introduce the following constant of
Dardé-Ervedoza \cite{Darde-Ervedoza-2019} and Lissy \cite{lissy2026optimal}:
\begin{equation} \label{Eq:Kappa*}
\kappa_\star = \frac{\Gamma\left( \frac{1}{4} \right)^4}{8\pi^3} \simeq 0.6966.
\end{equation}
We prove the following result.
\begin{theorem} \label{Thm:Main}
Let $L>0$ be fixed. There exists $C>0$ such that 
the following holds for all $T \in (0,1]$.
Given $v_T$ in $L_a^2(\sq)$,
there exist $g_-$ and $g_+$ in $L^2(0,T)$ such that the corresponding solution $v$
of \eqref{Eq:Heat}--\eqref{Eq:Heat-BC} satisfies 
\begin{equation} \label{Eq:Goal}
v(T,\cdot) = v_T(\cdot) \ \text{ in } \ (-L,L),
\end{equation}
and which are furthermore estimated by
\begin{equation} \label{Eq:Est-Controls}
\|g_-\|_{L^2(0,T)} + \|g_+\|_{L^2(0,T)} 
 \leq \frac{C}{T^{2}} \exp\left(\frac{\kappa_\star {L}^2}{T}\right) \|v_T\|_{L^2(\sq)}.
\end{equation}
\end{theorem}
Actually, we will prove a slightly stronger result. 
We introduce indeed the weight function $\lambda$ as follows:
\begin{equation} \label{Def:lambda}
\lambda(z) = \frac{\Im(z)^2}{4T},
\end{equation}
and consider the following weighted Bergman space where the area measure
is multiplied by $e^{-2\lambda}$:
\begin{equation} \label{De:H_C}
\Hsq := L^2_a(\sq; \, e^{-2\lambda}).
\end{equation}
Obviously, for fixed $T>0$, the vector spaces $\Hsq$ and $L^2_a(\sq)$ coincide, 
and the norms are equivalent, with moreover 
$\|\cdot\|_{\Hsq} \leq \|\cdot\|_{L^2_a(\sq)}$.
But the other inequality is not uniform in $T$, so that 
the following statement is a bit stronger.

\begin{manualtheorem}{1'} \label{Thm:1prime}
Let $L>0$ be fixed. There exists $C>0$ such that 
the following holds for all $T \in (0,1]$.
Given $v_T$ in $L_a^2(\sq)$,
there exist $g_-$ and $g_+$ in $L^2(0,T)$ such that the corresponding solution $v$
of \eqref{Eq:Heat}--\eqref{Eq:Heat-BC} satisfies \eqref{Eq:Goal}
and which are furthermore estimated by
\begin{equation} \label{Eq:Est-Controls2}
\|g_-\|_{L^2(0,T)} + \|g_+\|_{L^2(0,T)} 
 \leq \frac{C}{T^{2}} \exp\left(\frac{\kappa_\star {L}^2}{T}\right) \|v_T\|_{\Hsq}.
\end{equation}
\end{manualtheorem}
Theorem~\ref{Thm:1prime} allows us to give an alternative proof of the
aforementioned result by Dardé and Ervedoza \cite{Darde-Ervedoza-2019}
giving an upper estimate on the cost of the null-controllability of the heat
equation, with the tiny improvement that the constant in the exponential is 
actually attained. Note that this constant was recently proved to be optimal
by Lissy \cite{lissy2026optimal}.
\begin{theorem} \label{Thm:DEbis}
Let $L>0$ be fixed. There exists $C>0$ such that 
the following holds for all $T \in (0,1]$.
Given $v_0$ in $L^2(-L,L)$,
there exist $g_-$ and $g_+$ in $L^2(0,T)$ estimated by
\begin{equation} \label{Eq:Est-Controls-NC}
\|g_-\|_{L^2(0,T)} + \|g_+\|_{L^2(0,T)} 
 \leq \frac{C}{T^{2}} \exp\left(\frac{\kappa_\star {L}^2}{T}\right) \|v_0\|_{L^2(-L,L)},
\end{equation}
such that the corresponding solution $v$
of \eqref{Eq:Heat}, \eqref{Eq:Heat-BC} and initial condition
\begin{equation} \label{Eq:Heat-IC-NH}
v(0,x) = v_0(x) \ \text{ for } \ x \in (-L,L),
\end{equation}
satisfies 
\begin{equation} \label{Eq:Goal-NC}
v(T,\cdot) = 0 \ \text{ in } \ (-L,L).
\end{equation}
\end{theorem}
\begin{remark}
Notice that in all of the above results, we may also consider the case $T>1$
by replacing the $T$-dependent factor with $1$, 
meaning that the size of the control can be made bounded for $T> 1$.
We may indeed put the control to $0$ in the time interval 
$[0,T-1/2]$ and use the remaining interval for the control 
(also using the decay of the $L^2$ norm for the uncontrolled heat
equation when considering Theorem~\ref{Thm:DEbis}).
\end{remark}
We now give a brief review of some previous papers connected to this study. 
This paper is at the junction of two problems in PDE control theory
that have been broadly studied since the 1970s, namely the cost for small times
of the null-controllability of the heat equation, and the determination of its
reachable states by means of boundary control. \par
The study of the reachable space for the heat equation goes
back to Fattorini \cite{fattorini1978reachable} 
and Seidman \cite{seidman1979time}. 
A more accurate description was later given by Ervedoza \& Zuazua \cite{Ervedoza-Zuazua-2011}. 
More recently, several papers gave a description (or at least, inclusions) of this space 
relying on holomorphic extensions of the state, starting with 
the work of Martin, Rosier \& Rouchon \cite{MRR-2016}, and then with more 
and more precision in articles by Dardé \& Ervedoza \cite{Darde-Ervedoza-2018}, 
Hartmann, Kellay \& Tucsnak \cite{hartmann2020reachable}, 
Kellay, Normand \& Tucsnak \cite{Kellay-Normand-Tucsnak} and Orsoni $\cite{Orsoni}$. 
In particular, the last two references describe {\it as an equality} 
the reachable space as the sum of two Bergman spaces on sectors (see the domains
$D_1$ and $D_2$ below), with the slight difference that these spaces are weighted 
in the case \cite{Kellay-Normand-Tucsnak} (which is natural as we will see).
Finally, Hartmann \& Orsoni \cite{Hartmann-Orsoni} gave a description of 
this sum of (unweighted) Bergman spaces, 
confirming a conjecture of \cite{hartmann2020reachable} 
and yielding the simple and elegant form \eqref{Eq:Tucsnak_et_al}. \par
We now discuss the cost for small times of the null-controllability.
That the heat equation is null-controllable by means of boundary controls 
for any positive time is a classical result due to Fattorini \& Russell
\cite{Fattorini-Russell}. The problem of determining the cost of this 
null-controllability, that is to estimate the size of the boundary controls
(for instance in $L^2(0,T)$) in terms of the size of the initial condition 
(for instance $\|u_0\|_{L^2}$) was first investigated 
by Seidman \cite{Seidman-1984} (upper estimate) and G\"uichal \cite{Guichal-1985}
(lower estimate). Their results prove in particular that the best 
null-controllability constant $C_H(T,L)$, that is, the lowest value $K$ for which
given $u_0 \in L^2(-L,L)$, there exist $g_-$ and $g_+$ driving $u_0$ to $0$ 
in time $T$, satisfying 
$\|g_{-}\|_{L^2(0,T)} + \|g_{+}\|_{L^2(0,T)} \leq K \|u_0\|_{L^2(-L,L)}$, 
follows an exponential behavior of the type $\exp\left( \mathcal{O}(1) \frac{L^2}{T} \right)$:
\begin{equation*}
\kappa_- := \liminf_{T \rightarrow 0^+} \frac{T}{L^2} \log C_H(T,L) >0
\text{ and }
\kappa_+ := \limsup_{T \rightarrow 0^+} \frac{T}{L^2} \log C_H(T,L) <+\infty.
\end{equation*}
Many studies followed to improve the estimates of the constants $\kappa_-$ and $\kappa_{+}$ above.
In particular, Miller \cite{Miller-Cost-2006} gave an estimate 
$\kappa_+ \leq 2 \left( \frac{36}{37} \right)^2$, that was later improved 
by Tenenbaum \& Tucsnak \cite{Tenenbaum-Tucsnak} as $\kappa_+ \leq \frac{3}{4}$.
Then the best upper estimate $\kappa_+ \leq \kappa_\star$ was obtained 
by Dardé \& Ervedoza \cite{Darde-Ervedoza-2019}. 
Concerning lower estimates, G\"uichal's result was improved by 
Miller \cite{Miller04} who proved $\kappa_- \geq \frac{1}{4}$, and then by
Lissy \cite{Lissy-2015} ($\kappa_- \geq \frac{1}{2}$). This culminated with 
the recent work by Lissy \cite{lissy2026optimal} who proved 
$\kappa_- \geq \kappa_\star$ and consequently $\kappa_-=\kappa_+=\kappa_\star$. \par
The null-controllability amounts to saying that $\mathcal{S}_T:=\{S_T[u_0], \ u_0 \in L^2(-L,L)\} \subset \mathcal{R}_T$, where $(S_t)_{t \geq 0}$  is the heat semigroup associated with \eqref{Eq:Heat} and homogeneous Dirichlet boundary 
conditions. 
However, $\mathcal{S}_T$ is much smaller than $\mathcal{R}_T$ (for instance, it is composed of entire functions).
We refer in particular to Kellay, Normand and Tucsnak \cite[Section 6]{Kellay-Normand-Tucsnak} for a discussion
on the connections between the reachable space and the cost of null-controllability.
The above references on the cost of the null-controllability leave therefore 
open the question of the cost of reaching an attainable state.
The purpose of this paper is to fill this gap.  \par
We finish this introduction by noticing that, as in Theorem~\ref{Thm:DEbis}, the constant $\kappa_\star$ inside the exponential in Theorems~\ref{Thm:Main} 
and \ref{Thm:1prime} is optimal as a consequence of Lissy's work 
\cite{lissy2026optimal}. Indeed, he shows that driving the first eigenmode of the Laplacian $e_1:x \mapsto \cos\left( \frac{\pi x}{2L} \right)$ to $0$ has a cost at least
$\mathcal{O}(1) \exp\left( \frac{(\kappa_\star-\varepsilon)L^2}{T} \right)$ for any $\varepsilon>0$. So the cost of driving $0$ to $\exp(-\frac{\pi^2}{4L^2}T) e_1(x)$ is also of the same order.
%
%
%
%
%
%
%
%
%
%
%
%
\section{Proof of the main result}
%
%
%
%
%
%
\subsection{Definitions and notations}
%
%
%
%
\subsubsection{General notations}
Given an open set $\mathcal{O} \subset \C$, we will denote by $\mathscr{H}(\mathcal{O})$ 
the space of holomorphic functions on $\mathcal{O}$.
Given a positive continuous weight function $\mu:\mathcal{O} \rightarrow \R^{+*}$,
we will also define the weighted Bergman space
\begin{equation*}
L^2_a(\mathcal{O}; \, \mu) := L^2(\mathcal{O}; \, \mu(z) \,dA(z)) \cap \mathscr{H}(\mathcal{O}),
\end{equation*}
and simply write $L^2_a(\mathcal{O})$ for $L^2_a(\mathcal{O};\, dA(z))$.
Here $dA$ designates the area measure on $\C$. \par
The space $L^2_a(\mathcal{O};\mu)$ is closed in $L^2(\mathcal{O};\mu)$ 
and hence itself a Hilbert space.
This is very classical and due to the fact that for holomorphic functions 
the $L^2$ convergence implies the local convergence in $C^k$ spaces 
on smaller sets. \par
We will use the standard notations for Wirtinger derivatives:
for $f: \mathcal{O} \rightarrow \C$,
\begin{equation*}
\partial f := \frac{1}{2} \left( \frac{\partial}{\partial x} f 
    - i \frac{\partial}{\partial y} f\right)
\ \text{ and } \ 
\overline{\partial} f := \frac{1}{2} \left( \frac{\partial}{\partial x} f 
    + i \frac{\partial}{\partial y} f\right).
\end{equation*}
We recall that
\begin{equation} \label{Eq:Laplacian}
\partial \overline{\partial} = \frac{1}{4} \Delta,
\end{equation}
and the Cauchy-Green formulas (cf. e.g. \cite[p. 59]{MR507701}):
\begin{eqnarray} \label{Eq:Cauchy-Green}
\int_{\omega} f \, (\overline{\partial} g) = -\int_{\omega} (\overline{\partial}f) \,  g 
    + \frac{1}{2i} \int_{\partial \omega} fg\, dz, \\
\label{Eq:Cauchy-Green2}
\int_{\omega} f \, (\partial g)  = -\int_{\omega} (\partial f) \,  g 
    - \frac{1}{2i} \int_{\partial \omega} fg\, d\overline{z}.
\end{eqnarray}
Note that the Green formula is valid in Lipschitz domains, 
see \cite[Theorem 1.5.3.1]{grisvard1985elliptic}. \par
\ \par
\noindent
{\bf Convention on $\lesssim$.}
As a final notation, we will write inequalities $a \lesssim b$ to express the inequality
$a \leq C\, b$, up to a uniform constant $C>0$ which does not depend on the function 
or the point at stake, {\bf nor on $\mathbf{T}$}. 
It may however depend on $L$ which is considered fixed from the beginning. \par
%
%
%
%
%
%
%
%
\subsubsection{Complex domains}
\label{Subsubsec:Domains}
Let $L>0$.
We define $D_1, D_2 \subset \C$ as follows (see Figure~\ref{Fig:Domaines}):
\begin{gather}
\nonumber 
D_1 := \left\{ z \in \C \ \Big/ \ \Re(z)+L > |\Im(z)| \right\}, \\
\nonumber 
D_2 := \left\{ z \in \C \ \Big/ \ L-\Re(z) > |\Im(z)| \right\}, \\
\label{Def:Omega_and_square}
\Omega = D_1 \cup D_2 \ \text{ and } \ \sq = D_1 \cap D_2.
\end{gather}
\begin{figure}[ht]
\begin{center}
\hspace*{2.5cm}
\begin{tikzpicture}
    \begin{scope}
        \clip (-2,-3) -- (1,0) -- (-2,3);
        \fill[black!5] (-2,-3) rectangle (2,3);
    \end{scope}
    \begin{scope}
        \clip (2,-3) -- (-1,0) -- (2,3);
        \fill[black!5] (-2,-3) rectangle (2,3);
    \end{scope}
    \draw[black, dotted] (0,-2.8) -- (0,3);
    \draw[black, dotted] (-3.5,0) -- (3.5,0);

    \fill[black!5] (4,2) rectangle (4.8,2.5);

    \draw[blue, thick] (1.5,-2.5) -- (-1,0) -- (1.5,2.5);
    \draw[blue!70!black, thick, dashed] (1.5,-2.5) -- (2,-3);
    \draw[blue!70!black, thick, dashed] (1.5,2.5) -- (2,3);

    \draw[green!70!black, thick] (-1.5,-2.5) -- (1,0) -- (-1.5,2.5);;
    \draw[green!70!black, thick, dashed] (-1.5,-2.5) -- (-2,-3);
    \draw[green!70!black, thick, dashed] (-1.5,2.5) -- (-2,3);

    \node[blue] at (1.7,0.5) {$D_1$};
    \node[green!70!black] at (-1.7,0.5) {$D_2$};
    \node[right] (label) at (0.05,2.15) {$\sq$};
    \draw[->,>=stealth,shorten >=2pt] (0.1,2) to[bend right=30] (0.2,0.2);
    \node[black] at (5.2,2.25) {$\Omega$};
    \node[black] at (0,-3) {$\scriptscriptstyle i\R$};
    \node[black] at (3.7,0.02) {$\scriptscriptstyle \R$};
\end{tikzpicture}
\end{center}
\caption{Complex domains}
\label{Fig:Domaines}
\end{figure}
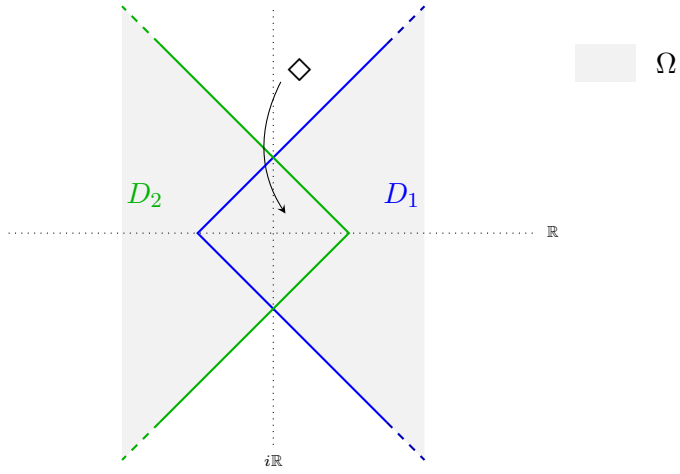
\ \par
The following subdomains will also be helpful:
\begin{equation} \label{Eq:Subdomains}
\sql := \sq \cap \{z \in \C, \ \Re(z) <0\}
\ \text{ and } \ 
\sqr := \sq \cap \{z \in \C, \ \Re(z) > 0\}.
\end{equation}
Finally, we will call $\NS$ (respectively $\SN$) the directed straight line from $iL$ to $-iL$
(resp. from $-iL$ to $iL$). We call it $\mathcal{M}$ when the direction has no importance. \par
%
%
%
%
%
%
%
\subsubsection{Weights and spaces}
For a given $T>0$, we first define the functions $\varphi_1$ and $\varphi_2$
on $\Omega$ as follows:
\begin{equation} \label{Def:Phi-i}
\varphi_{1}(z) = \frac{ (L+z)^2 }{4T}
\ \text{ and } \ 
\varphi_{2}(z) = \frac{ (L-z)^2 }{4T}
    \ \text{ for } z \in \Omega.
\end{equation}
We note that $\varphi_1$ and $\varphi_2$ are holomorphic in $\Omega$. 
The functional spaces $L_a^2 \left(D_k; \, e^{2 \Re(\varphi_k)}\right)$ on $D_k$ 
are at the core of the discussion, though most of the time not used explicitly. 
%
\par
%
%
%
%
%
\subsubsection{Main operator}
For $k=1,2$, we introduce $B_k$ the Bergman projector in $L^2(D_k)$ 
(without weight), that is, the orthogonal projector in $L^2(D_k)$ onto $L_a^2(D_k)$. \par
The principal operator that we are studying in this paper is the following:
\begin{equation} \label{Eq:Principal-operator}
\mathcal{T} : \left\{ \begin{array}{l}
\displaystyle 
L^{2}_a \left( D_1 \right) 
    \times L^{2}_a\left( D_2 \right)
\longrightarrow
\Hsq , \bigskip \\
\displaystyle 
(u_1, u_2) \longmapsto e^{-\varphi_1}{u_1}_{|\sq} + e^{-\varphi_2}{u_2}_{|\sq}.
\end{array} \right. 
\end{equation}
The main goal is to prove that $\mathcal{T}$ is surjective and to estimate 
a right-inverse for $\mathcal{T}$.
%
%
%
%
%
%
%
%
\subsection{Main point of the proof}
\label{Subsec:Main-Lemmas}
In this section, we introduce the main statement in the proof of Theorem~\ref{Thm:Main}. \par
\ \par
First, we describe the adjoint operator of $\mathcal{T}$.
\begin{lemma} \label{Lem:Dual-T}
The adjoint operator of $\mathcal{T}$ defined in \eqref{Eq:Principal-operator}
is given by
\begin{equation} \label{Eq:Dual-T}
\forall v \in \Hsq, \ \
\mathcal{T}^*(v) = 
\left(
B_1 \circ E_1\left[ e^{-2 \lambda - \overline{\varphi_1}} v \right] , \ 
B_2 \circ E_2\left[ e^{-2 \lambda - \overline{\varphi_2}} v \right] 
\right),
\end{equation}
where for $k=1,2$, $E_k$ designates the operator 
$L^{2}(\sq) \rightarrow L^{2}(D_k)$ that extends a function defined in $\sq$
by $0$ in $D_k \setminus \sq$.
\end{lemma}
\begin{proof}[Proof of Lemma~\ref{Lem:Dual-T}]
Let 
$v \in \Hsq$ and
$(u_1,u_2) \in L^{2}_a \left(D_1\right) \times L^{2}_a\left(D_2\right)$.
Using the definition of $B_k$, that $u_k \in L^2_{a}(D_k)$, $k=1,2$,
and that $\lambda$ is real-valued, we have
\begin{align*}
\left\langle \mathcal{T}(u_1,u_2), v 
    \right\rangle_{\Hsq}
&= \sum_{k=1}^{2} 
\left\langle e^{- 2\lambda} \, e^{-\varphi_k} \, u_k  , \, v \right\rangle_{L^2(\sq)} \\
&= \sum_{k=1}^{2} 
    \left\langle u_k , \, 
     e^{- 2\lambda - \overline{\varphi_k}} \,  v \right\rangle_{L^2(\sq)} \\
&= \sum_{k=1}^{2} 
    \left\langle u_k , \, 
    E_k\left( e^{- 2\lambda - \overline{\varphi_k}} \,  v\right) \right\rangle_{L^2(D_k)} \\
&= \sum_{k=1}^{2} 
    \left\langle u_k , \, 
    B_k \circ E_k\left( e^{- 2\lambda - \overline{\varphi_k}} \,  v\right) \right\rangle_{L^2(D_k)}.
\end{align*}
Since for each $k=1,2$,
$B_k \circ E_k\left( e^{-2\lambda - \overline{\varphi_k}} v\right)$
belongs to $L^2_a(D_k)$, this ends the proof.
\qedhere
\end{proof}
\noindent
The following proposition is the core of the proof.
\begin{proposition} \label{Prop:Observability}
Given $L>0$, there exists $C_1>0$, such that for $T \in (0,1]$,
for $v \in \Hsq$, one has:
\begin{equation} \label{Eq:Observability}
\| v \|_{\Hsq}
\leq \frac{C_1}{T^2}\, \exp\left( \kappa_\star \frac{{L}^{2}}{T} \right)
\| \mathcal{T}^*(v)\|_{L^{2}(D_1) \times L^{2}(D_2)} .
\end{equation}
\end{proposition}

\noindent
The proof of Proposition~\ref{Prop:Observability} is given in Section~\ref{Sec:Main-proof}.
%
%
%
%
%
%
%
\subsection{Proof of Theorems~\ref{Thm:Main} and \ref{Thm:1prime}}
\label{Subsec:Proof-Theorem}
Here we give the proof of Theorem~\ref{Thm:1prime} (and consequently of 
Theorem~\ref{Thm:Main}) based on the statements
of Subsection~\ref{Subsec:Main-Lemmas}. 
This can be seen as a variant of the proof by Hartmann, Kellay and Tucsnak \cite{hartmann2020reachable}. \par
\ \par
First, Proposition~\ref{Prop:Observability} shows that $\mathcal{T}$ is surjective from
$L^{2}_a \left(D_1\right) \times L^{2}_a\left(D_2\right)$
to $\Hsq$ and admits a continuous right-inverse $\mathcal{S}$ that moreover satisfies:
\begin{equation*}
\vvvert \mathcal{S} \vvvert \leq \frac{C_1}{T^2} \, 
\exp\left( \kappa_\star \frac{L^2}{T} \right),
\end{equation*}
as an operator 
$\Hsq \rightarrow
L^{2}_a \left(D_1\right) \times L^{2}_a\left(D_2\right)$. \par
It follows that for all $u_T \in \Hsq$, there exists
$(u_1,u_2) \in L^{2}_a \left(D_1\right) \times L^{2}_a\left(D_2\right)$
such that for $k=1,2$,
\begin{equation*}
\left\|  u_k \right\|_{L^{2}_a \left(D_k\right) } 
\leq \frac{C_1}{T^2}\, \exp\left( \kappa_\star \frac{{L}^{2}}{T} \right)
\| u_{T} \|_{H_{\sq}},
\end{equation*}
and 
\begin{equation} \label{Eq:Final-Decomposition}
{u_{T}}_{|\sq} =  e^{-\varphi_1}{u_1}_{|\sq} + e^{-\varphi_2}{u_2}_{|\sq}.
\end{equation}
In other words, for all $u_T \in \Hsq$, there exists
$(\tilde{u_1},\tilde{u_2}) \in 
L^{2}_a \left(D_1; \, e^{2 \Re(\varphi_1)}\right) \times 
L^{2}_a\left(D_2; \, e^{2 \Re(\varphi_2)}\right)$
such that for $k=1,2$,
\begin{equation*}
\left\|  \tilde{u}_{k} \right\|_{L^{2}_a \left(D_k;\, e^{2 \Re(\varphi_k)}\right) } 
\leq \frac{C_1}{T^2}\, \exp\left( \kappa_\star \frac{{L}^{2}}{T} \right)
\| u_{T} \|_{H_{\sq}},
\end{equation*}
and 
\begin{equation} \label{Eq:Final-Decomposition-2}
{u_{T}}_{|\sq} =  \tilde{u_1}_{|\sq} + \tilde{u_2}_{|\sq}.
\end{equation}
\ \par
We can then conclude by using a result by Aikawa, Hayashi \& Saitoh \cite{Aikawa-Hayashi-Saitoh}.
Precisely, we use the following statement, see \cite[Theorem 6, p. 137]{Saitoh-Book}.
\begin{theorem*}[Aikawa-Hayashi-Saitoh]
The integral transform
\begin{equation} \label{Eq:IntegralTransform}
\mathcal{I}_T[F](z)= \frac {1}{T} \int_{0}^{T} F(\xi) 
\frac{z \exp \left[ \frac{-z^ {2}}{4(T-\xi)} \right]}{2\sqrt{\pi}(T-\xi)^{3/2}}\, 
\xi \, d\xi 
\end{equation}
for functions $F$ satisfying
$\int_{0}^{T} |F(\xi)|^{2} \, \xi \, d\xi < \infty$,
generates an isometry between $L^2((0,T), \frac{\xi}{T} d\xi)$
and the space of holomorphic functions on 
$L^2_a\left(Q_0; \exp\left( \Re\left\{ \frac{z^2}{2T} \right\} \right) \, dA(z)\right)$,
where 
\begin{equation*}
Q_0 := \left\{ z \in \C \ \Big/ \ \Re(z) > |\Im(z)| \right\}.
\end{equation*}
\end{theorem*}
\ \par
Now for $h \in L^2(0,T)$, the solution of 
\begin{equation} \nonumber 
\left\{ \begin{array}{l}
\partial_t u -\partial^2_{xx} u =0 \text{ for } (t,x) \in (0,T) \times \R^+ , \\
u(t,0) = h(t) \text{ on } (0,T) \\
u(0,x) = 0 \text{ on } \R^+ ,
\end{array} \right. 
\end{equation}
is given by the following convolution, see e.g. \cite[Eq. (7), p. 129]{Saitoh-Book}:
\begin{equation} \label{Eq:HeatAsConvolution}
u (t,x)= \int_{0}^{t} h(\xi) 
\frac{x \exp \left[ \frac{-x^ {2}}{4(t-\xi)} \right]}{2\sqrt{\pi}(t-\xi)^{3/2}}\, d\xi,
\end{equation}
Adapting the above theorem to $D_1$ and $D_2$ and applying it
to $\tilde{u}_1/T$ and $\tilde{u}_2/T$,
we find that there exist 
$F_1$ and $F_2$ in $L^2( (0,T), \xi\, d\xi)$
such that $\mathcal{I}_{T,D_1}[F_1] = \tilde{u}_1/T$ 
and $\mathcal{I}_{T,D_2}[F_2] = \tilde{u}_2/T$, with for $k=1,2$,
\begin{equation*}
\|F_k\|_{L^2((0,T), \frac{\xi}{T} d\xi)} 
= \|\mathcal{I}_{T,D_k}[F_k]\|_{L^2_a(D_k, e^{2 \Re(\varphi_k)})}
= \frac{1}{T}
    \left\| \widetilde{u}_k \right\|_{L^{2}_a \left(D_k, \, e^{2\Re(\varphi_k)}\right)}.
\end{equation*}
Correspondingly, comparing \eqref{Eq:IntegralTransform} and \eqref{Eq:HeatAsConvolution},
we see that the solutions $v_1$ and $v_2$ of
\begin{equation*}
\left\{ \begin{array}{l}
\partial_t v_1 -\partial^2_{xx} v_1 =0 \text{ in } (0,T) \times (-L,+\infty), \\
v_1(t,-L) = t \, F_1(t) \text{ on } (0,T) \\
v_1(0,x) = 0 \text{ on } (-L,+\infty),
\end{array} \right. 
\ \ 
\left\{ \begin{array}{l}
\partial_t v_2 -\partial^2_{xx} v_2 =0 \text{ in } (0,T) \times (-\infty,L), \\
v_2(t,L) = t \, F_2(t) \text{ on } (0,T) \\
v_2(0,x) = 0 \text{ on } (-\infty,L),
\end{array} \right. 
\end{equation*}
reach $\tilde{u}_{1|(-L,+\infty)}$ and $\tilde{u}_{2|(-\infty,L)}$ at time $T$,
respectively. Calling $G_k(t) = t F_k(t)$, we have the estimates on $G_1$ and $G_2$:
for $k=1,2$,
\begin{equation*}
\| G_k \|_{L^2(0,T)} 
= \| G_k (\cdot) / \sqrt{\cdot} \|_{L^2(0,T;\, \xi\, d\xi)} \leq
T^{1/2} \| F_k \|_{L^2(0,T;\, \xi\, d\xi)} = 
\left\| \widetilde{u_k} \right\|_{L^{2}_a \left(D_k, \, e^{2\Re(\varphi_k)}\right)} .
\end{equation*}
Now by superposition principle, $v:=v_1 + v_2$ satisfies
\begin{equation*}
\left\{ \begin{array}{l}
\partial_t v -\partial^2_{xx} v =0 \text{ for } (t,x) \in (0,T) \times (-L,L), \\
v(t,-L) = g_{-}(t) \text{ on } (0,T) \\
v(t,L) = g_{+}(t) \text{ on } (0,T) \\
v(0,x) = 0 \text{ on } (-L,L),
\end{array} \right. 
\end{equation*}
with
\begin{equation*}
g_{-}(t) := G_{1}(t) + v_{2}(t,-L)
\ \text{ and } \ 
g_{+}(t) := G_{2}(t) + v_{1}(t,L).
\end{equation*}
By using $L^1 \times L^2$ convolution estimate 
(relying on \eqref{Eq:HeatAsConvolution}), we see that for $k=1,2$,
\begin{equation*}
\left\| v_k \left(t,(-1)^{k+1}L\right) \right\|_{L^2(0,T)} 
    \lesssim \| G_k \|_{L^2(0,T)}.
\end{equation*}
This gives the desired estimate on $g_{-}$ and $g_{+}$.
\hfill $\Box$
%
%
%
%
%
%
%
%
\subsection{The case of the null controllability: proof of Theorem~\ref{Thm:DEbis}}
\label{Subsec:Null-Controllability}
We will use this following elementary lemma.
\begin{lemma} \label{Lem:Est-Free-Solution}
Let $L>0$. Define $\lambda$ as in \eqref{Def:lambda}.
For any $T>0$ the following holds.
For $u_0 \in L^{2}(-L,L)$, we define 
\begin{equation} \label{Eq:Extension-u0}
E[u_{0}](x) = u_{0}(x) \ \text{ in } \ (-L,L) 
\ \text{ and } \ 
E[u_{0}](x) = 0 \ \text{ in } \ \R \setminus (-L,L) .
\end{equation}
Let $\left(e^{t\partial^2_{xx}}\right)_{t \geq 0}$ the heat semigroup on $L^2(\R)$.
Then $e^{T\partial^2_{xx}} E[u_0]$ can be extended as an entire function 
on $\C$, whose restriction to $\sq$ moreover satisfies
\begin{equation} \label{Eq:Est-Free-Solution}
\left\| e^{T\partial^2_{xx}} E[u_0] \right\|_{\Hsq} 
\lesssim \|u_0\|_{L^2(-L,L)}.
\end{equation}
\end{lemma}
\begin{proof}[Proof of Lemma~\ref{Lem:Est-Free-Solution}]
We rely on the convolution formula 
\begin{equation*}
\forall z \in \C, \ \ 
e^{T\partial^2_{xx}} E[u_0] (z)  = \frac{1}{\sqrt{4\pi T}} 
    \int_{-L}^{L} u_0(x) \exp\left(-\frac{(z-x)^{2}}{4T}\right) \, dx,
\end{equation*}
which expresses $e^{T\partial^2_{xx}} E[u_0]$ as an entire function. \par
Now for $z = a + ib \in \sq$ and $x \in (-L,L)$, we have
\begin{eqnarray*}
\exp\left( -\lambda(z) \right) \exp\left(-\frac{(z-x)^{2}}{4T}\right)
&=& \exp\left( -\frac{b^2}{4T} -\frac{\left[ (a-x) + ib \right]^2}{4T} \right) \\
&=& \exp\left( -\frac{(a-x)^2}{4T} - i\frac{2b(a-x)}{4T} \right).
\end{eqnarray*}
Hence, for fixed $b \in (-L,L)$, we have
\begin{equation*}
\exp\left( -\lambda(a+ib) \right)e^{T\partial^2_{xx}} E[u_0] (a+ib) =\frac{1}{\sqrt{4\pi T}} 
 \int_{-L}^{L} u_0(x) \exp\left( -\frac{(a-x)^2}{4T} - i\frac{2b(a-x)}{4T} \right) \, dx.
\end{equation*}
Since the function 
$y \mapsto \frac{1}{\sqrt{4\pi T}} \exp\left( -\frac{y^2}{4T} - i\frac{2by}{4T} \right)$
has an $L^1$ norm equal to $1$, the convolution gives that each $b$-slice of 
$\exp\left( -\lambda(z) \right) e^{T\partial^2_{xx}} E[u_0] (z)$ in $\sq$ 
can be estimated in $L^2$ by $\|u_0\|_{L^2(-L,L)}$.
It follows that \eqref{Eq:Est-Free-Solution} holds.
\end{proof}
\begin{proof}[Proof of Theorem~\ref{Thm:DEbis}]
Now given $u_0 \in L^2(-L,L)$, we introduce $E[u_0]$ as in \eqref{Eq:Extension-u0}, 
and then $u_T := e^{T\partial^2_{xx}} E[u_0] \in \Hsq$. 
Applying Lemma~\ref{Lem:Est-Free-Solution} and Theorem~\ref{Thm:1prime} to this $u_{T}$, we find controls $\check{g}_-$ and $\check{g}_+$ driving $0$ to $u_T$, 
satisfying estimates as in \eqref{Eq:Est-Controls-NC}, and a corresponding solution 
$v$ of \eqref{Eq:Heat}--\eqref{Eq:Heat-BC}. 
Then $-v+e^{t\partial^2_{xx}} E[u_0] $ drives $u_0$ to $0$ in $(-L,L)$, and 
we introduce ${g}_-$ and ${g}_+$ as its trace at $x = \pm L$.
That the $e^{t\partial^2_{xx}} E[u_0]$-part of $g_-$ and $g_+$ also satisfy
estimates as in \eqref{Eq:Est-Controls-NC} is elementary, since for instance
the energy equality for the heat equation gives 
$\|e^{t\partial^2_{xx}} E[u_0]\|_{L^2(0,T; H^1(\R))} \leq (1+T^{1/2})\|u_0\|_{L^2}$.
Theorem~\ref{Thm:DEbis} follows.
\end{proof}
%
%
%
%
%
%
%
%
%
%
%
%
%
%
%
\section{Main part of the proof: Proposition~\ref{Prop:Observability}}
\label{Sec:Main-proof}
We now tackle the proof of Proposition~\ref{Prop:Observability}, which,
as we mentioned before, is the principal part of the proof.
We begin with some preliminary material.
%
%
%
%
%
%
\subsection{Poisson equation and Bergman projectors in $L^2(D_k)$}
Following Bell \cite[Theorem 15.2]{Bell2015}, which is written in the context of smooth 
bounded domains, we can give an explicit form of the Bergman projector in $D_k$
(for particular data) as follows. In particular, \eqref{Eq:Bergman-on-Di} below
is referred to in \cite{Bell2015} as {\it Spencer's formula}. \par
As before we denote $B_k$ the Bergman projector 
$L^{2}(D_k) \rightarrow L_a^{2}(D_k)$, 
and define its complement projector $B_k^{\perp} := \Id - B_k$. 
To simplify the presentation, we first focus on the case of the quadrant
\begin{equation*}
Q:= \left\{ z \in \C \ \Big/ \ \Re(z) \geq 0, \ \Im(z) \geq 0 \right\},
\end{equation*}
and study the Bergman projector $B$ and its complement $B^\perp =\Id - B$ there,
the transposition of the result to $D_1$ and $D_2$ being straightforward. \par
\ \par
The first step is to consider Poisson's equation in the quadrant. 
This is described in the following lemma.
\begin{lemma} \label{Lem:Poisson}
For any $F \in L^{2}(Q; \R^2)$, there exists a unique function $u$ in
\begin{equation} \label{Eq:DefH1Point}
\dot{H}^1_0(Q) : = \left\{ u \in L^1_{loc}(Q) \ / \ \nabla u \in L^2(Q)
\text{ and } u_{|\partial Q} =0
 \right\}
\end{equation}
of
\begin{equation} \label{Eq:PoissonQ}
\Delta u = \div F \text{ in } Q, \ \ u=0 \text{ on } \partial Q.
\end{equation}
This solution satisfies
\begin{equation} \label{Eq:Est-Poisson-0}
\| \nabla u \|_{L^{2}(Q)} \leq \| F \|_{L^2(Q)},
\end{equation}
and, given $K$ a compact subset of $\overline{Q}$, there is a constant $C_K >0$ 
such that this function satisfies the estimate:
\begin{equation} \label{Eq:Est-Poisson-1}
\| u \|_{H^{1}(K)} \leq C_K \| F \|_{L^2(Q)}.
\end{equation}
\end{lemma}
Note that for $u \in L^1_{loc}(Q)$ with $\nabla u \in L^2(Q)$, the trace 
on $\partial Q$ is well-defined, since by a local Poincaré inequality, one has $u \in H^1(Q \cap B(0,R))$ for any $R>0$. \par
Lemma~\ref{Lem:Poisson} is proved in Section~\ref{Sec:Technical-proofs}. 
It has the following counterpart in $\Omega$, also proved therein.
\begin{lemma} \label{Lem:Poisson-Omega}
For any $F \in L^{2}(\Omega; \, \C)$, there exists a unique $u$ in 
\begin{equation*}
\dot{H}^1_0(\Omega) := \left\{ u \in L^2_{loc}(\Omega; \, \C)  \ /\ 
\nabla u \in L^2(\Omega) \text{ and } u_{|\partial \Omega} =0 \right\},
\end{equation*}
such that
\begin{equation} \label{Eq:Poisson}
\Delta u = \overline{\partial} F \text{ in } \Omega 
\text{ and } u=0 \text{ on } \partial \Omega.
\end{equation}
This solution satisfies
\begin{equation} \label{Eq:Est-Poisson-Omega}
\| \nabla u \|_{L^{2}(\Omega)} \leq \| F \|_{L^2(\Omega)},
\end{equation}
and for any compact set $K \subset \overline{\Omega}$, there is a constant $C_K>0$
such that this solution satisfies
\begin{equation} \label{Eq:Est-Psi-Hat}
\| u \|_{H^1(K)} 
\leq C_K \| F \|_{L^2(\Omega)} .
\end{equation}
\end{lemma}
Lemma~\ref{Lem:Poisson} allows one to define the Poisson operator $\Delta_Q^{-1}$ 
which maps $F \in \div L^2$ to $u$ in $\dot{H}^1_0(Q)$, and which we now use.
We have indeed the following lemma giving an explicit expression for the complement 
projector $B^{\perp}$ of $B$, and consequently for $B$ itself.
\begin{lemma} \label{Lem:Bergman-on-Di}
For any $v \in L^2(Q)$, the complement Bergman projector in $Q$ is given by:
\begin{equation} \label{Eq:Bergman-on-Di}
B_{Q}^{\perp}(v) 
= 4 \partial \Big( \Delta^{-1}_{Q} \left[\, \overline{\partial} v\right] \Big).
\end{equation}
\end{lemma}
Lemma \ref{Lem:Bergman-on-Di} is also proved in Section~\ref{Sec:Technical-proofs}.
%
%
%
%
%
%
%
\subsection{Traces of functions of $L^2_a(\mathfrak{\sq})$}
We state as a lemma a quite simple property of functions in $L^2_a(\sq)$.
\begin{lemma} \label{Lem:Traces-Bergman}
Any function in $L^2_a(\sq)$ admits traces in $H^{-1/2}(\partial \sq)$.
More generally, if $u \in L^2(\sq)$ satisfies $\overline{\partial} u \in L^2(\sq)$,
then $u$ admits a trace in $H^{-1/2}(\partial \sq)$, satisfying
\begin{equation*}
\| u_{|\partial \sq}\|_{H^{-1/2}(\partial \sq)} 
\lesssim \| u \|_{L^{2}(\sq)} + \| \overline{\partial} u \|_{L^{2}(\sq)}.
\end{equation*}
The same is true on $\sql$ and $\sqr$.
\end{lemma}
\begin{proof}[Proof of Lemma~\ref{Lem:Traces-Bergman}]
A way to see this is to use the correspondence between a complex-valued function $f$ on $\sq$ 
and the vector field on $\sq$:
\begin{equation*}
\vec{f} := \begin{pmatrix}
    \Re(f) \\ -\Im(f)
\end{pmatrix}.
\end{equation*}
Then we have the correspondence:
\begin{equation*}
2 \overline{\partial} f = \div\left(\vec{f}\, \right) - i \curl\left(\vec{f}\, \right)
\ \text{ and } \
\int_{\partial\sq} f(z)\, dz 
= \int_{\partial\sq} \left[ (\vec{f}\cdot\tau) + i(\vec{f}\cdot n)  \right] \, ds.
\end{equation*}
\par
\noindent
It is classical that divergence-free $L^2$ vector fields admit $H^{-1/2}$ normal traces, 
while curl-free $L^2$ vector fields admit $H^{-1/2}$ tangent traces (by mere rotation).
This fact is true in Lipschitz domains, see e.g. 
\cite[Chapter IV, Sections 3.2 and 4.2]{BoyerFabrie2012}. 
Now the case when the divergence and the curl are not zero but a function in $L^2(\sq)$ 
can be treated by the same proof: see formulae \cite[(IV.10) and (IV.13)]{BoyerFabrie2012}.
The conclusion follows. Another possibility is to rely on \eqref{Eq:Cauchy-Green}
which gives meaning to 
$\left\langle u, \, \Phi \right\rangle_{H^{-1/2}(\partial\sq) 
    \times H^{1/2}(\partial\sq)}$ 
by extension of $\Phi$ in $H^1(\sq)$, for any $\Phi$ in $H^{1/2}(\partial\sq)$.
\par
\end{proof}
%
%
%
%
%
\subsection{Proof of Proposition~\ref{Prop:Observability}}
We now proceed to the proof of Proposition~\ref{Prop:Observability} itself. 
We fix $L>0$ and $T>0$. \par
%
%
%
%
\noindent
\subsubsection{First step.} 
We let $v \in \Hsq$, which as previously mentioned is not different from
$L^{2}_a\left( \sq \right)$ as a vector space and has an equivalent norm for fixed $T$.
Due to the density of polynomials in $L^{2}_a\left(\sq\right)$ (see
Carleman \cite{carleman1927}), it is sufficient to prove \eqref{Eq:Observability} 
when $v$ is smooth. 
Since
\begin{equation*} 
\| \mathcal{T}^*(v) \|^2_{L^{2}(D_1) \times L^{2}(D_2)}
= \left\| B_1 \circ E_1\left[ e^{-2\lambda - \overline{\varphi}_1} \, v \right] 
                        \right\|^2_{L^{2}_a \left(D_1\right)}
+ \left\| B_2 \circ E_2\left[ e^{-2\lambda - \overline{\varphi}_2} \, v \right] 
                        \right\|^2_{L^{2}_a \left(D_2\right)},
\end{equation*}
our goal is to prove that for some constant $C_1>0$,
one has for $v \in \Hsq \cap C^{\infty}(\overline{\sq})$:
\begin{equation} \label{Eq:Central-Goal}
\| v \|_{\Hsq } \leq \frac{C_1}{T^2} \exp\left( \kappa_\star \frac{{L}^{2}}{T} \right)
\left(
\left\| B_1 \circ E_1\left[ e^{-2\lambda - \overline{\varphi_1}} \, v \right] 
    \right\|_{L^{2}_a(D_1)}
+ \left\| B_2 \circ E_2\left[ e^{-2\lambda - \overline{\varphi_2}} \, v \right] 
    \right\|_{L^{2}_a(D_2)}
\right).
\end{equation}
%
%
%
%
%
%
\subsubsection{Introduction of the $\A$ term.}
For $k=1,2$, we introduce, using Lemma~\ref{Lem:Poisson},
\begin{equation*}
\alpha_k := 4 \Delta^{-1}_{D_k} \Big[ 
    \overline{\partial} \, E_k[e^{-2\lambda - \overline{\varphi_k}} \,  v] 
\Big] 
\ \text{ in } \ D_k.
\end{equation*}
Note that as a consequence of Lemma~\ref{Lem:Bergman-on-Di}, one has
\begin{equation} \label{Eq:Decomp-v-i}
E_k \left[e^{-2\lambda - \overline{\varphi_k}} \, v\right] =
B_k \circ E_k\left[ e^{-2\lambda - \overline{\varphi_k}} \, v \right]
+ \partial\alpha_k
\ \text{ in } \ D_k.
\end{equation}
We have the following estimate for the functions $\alpha_k$.
\begin{lemma} \label{Lem:Est-Asymp-alphai}
For some constant $K>0$ depending only on $L$, $T$ and $v$ one has for $k=1,2$,
\begin{equation} \label{Eq:Est-Asymp-alphai}
|\alpha_k(z)| \leq \frac{K}{|z|^2} 
    \text{ for } z \in D_k \text{ with } |z| \geq 4L.
\end{equation}
\end{lemma}
\noindent
Lemma~\ref{Lem:Est-Asymp-alphai} is proved in 
Subsection~\ref{Subsec:Proof-Poisson-D}. \par
\ \par
Now we extend functions defined on $D_k$ (such as $\alpha_k$ or $B_k[w]$ for $w \in L^2(D_k)$)
or on $\sq$ (such as $v$) by $0$ in the rest of $\Omega$, and put a hat to indicate it:
\begin{eqnarray*}
\widehat{\alpha}_k(z) = \alpha_k(z) \ \text{ in } \ D_k
\ \text{ and } \ 
\widehat{\alpha}_k(z) = 0 \ \text{ in } \ \Omega \setminus D_k, \\
\widehat{B}_k[w](z) = {B}_k[w](z) \ \text{ in } \ D_k
\ \text{ and } \ 
\widehat{B}_k[w](z) = 0 \ \text{ in } \ \Omega \setminus D_k, \\
\widehat{v}(z) = v(z) \ \text{ in } \ \sq \ \text{ and } \ 
\widehat{v}(z) = 0 \ \text{ in } \ \Omega \setminus \sq.
\end{eqnarray*}
With this notation, we do no longer need the extension operator $E_k$ to define
$B_k[e^{-2\lambda - \overline{\varphi_k}} \, \widehat{v}]$; rather we should restrict
the function inside to $D_k$. This will be implicit in the sequel. \par
Note that, since for $k=1,2$, $\alpha_k$ belongs to $\dot{H}^{1}_{0}(D_k)$ 
(see Lemma~\ref{Lem:Poisson}), it follows that
$\widehat{\alpha}_1$ and $\widehat{\alpha}_2$ both belong to $\dot{H}^{1}_{0}(\Omega)$
and satisfy 
\begin{equation*}
\partial_z \widehat{\alpha}_k =
\left\{ \begin{array}{ll}
\partial_z {\alpha}_k & \ \text{ in } \ D_k, \\
0 & \ \text{ in } \ \Omega \setminus D_k.
\end{array} \right. 
\end{equation*}
They even belong to $H^1_0(\Omega)$ thanks to Lemma~\ref{Lem:Est-Asymp-alphai}.
On the other hand, $\hat{v}$ merely belongs to $L^2(\Omega)$. \par
\ \par
Next, we introduce
\begin{equation*}
\A(z) := \exp\left(\overline{\varphi}_1 \right) \, \widehat{\alpha}_1(z) 
        - \exp\left( \overline{\varphi_2} \right) \, \widehat{\alpha}_2(z)
\ \ \text{ in } \ \Omega.
\end{equation*}
Since $z \mapsto \exp\left(\overline{\varphi}_1 \right)$ 
and $z \mapsto \exp\left( \overline{\varphi_2} \right)$ 
are antiholomorphic, we see that for $z \in \Omega$,
\begin{equation} \nonumber 
\partial \A(z) 
= \exp\left(\overline{\varphi}_1 \right) \, \partial \widehat{\alpha}_1(z) 
        - \exp\left( \overline{\varphi_2} \right) \, \partial \widehat{\alpha}_2(z).
\end{equation}
Now for $k=1,2$, we decompose $e^{-2\lambda - \overline{\varphi_k}} \, \widehat{v}$ 
in $\Omega$, using \eqref{Eq:Decomp-v-i} and the above definitions:
\begin{equation} \label{Eq:Decomp-v}
e^{-2\lambda - \overline{\varphi_k}} \, \widehat{v}
 = \widehat{B}_k \left[e^{-2\lambda - \overline{\varphi_k}} \, \widehat{v} \right] 
    + \partial \widehat{\alpha}_k
\ \ \text{ in } \ \Omega.
\end{equation}
Multiplying \eqref{Eq:Decomp-v} 
by $\exp\left(\overline{\varphi}_1 \right)$ (for $k=1$)
and by $\exp\left(\overline{\varphi_2}\right)$ (for $k=2$) 
and subtracting the results, we get
\begin{multline} \label{Eq:MainEqA}
\partial \A(z) 
= \Ba(z) \text{ in } \Omega, \\
\text{ with }
\Ba(z) :=- \exp\left(\overline{\varphi}_1 \right) 
    \widehat{B}_1 [e^{-2\lambda - \overline{\varphi}_1} \widehat{v}](z) 
 + \exp\left( \overline{\varphi_2} \right) 
    \widehat{B}_2 [e^{-2\lambda - \overline{\varphi}_2}\widehat{v}](z).
\end{multline} \par
The next two paragraphs aim at estimating $\mathfrak{a}$ based on this equation.
%
%
%
%
%
%
\subsubsection{The Dardé-Ervedoza multiplier.} 
\noindent
Applying $\overline{\partial}$ to \eqref{Eq:MainEqA}, we observe that
$u=\A$ satisfies \eqref{Eq:Poisson} with $F$ given by $4\Ba$.
But the problem is that $F$ does not necessarily belong to $L^2(\Omega)$; 
it can actually grow like the inverse of a Gaussian at infinity!
To address this issue, we will use a complex multiplier introduced by
Dardé and Ervedoza \cite{Darde-Ervedoza-2019}, 
and described in the lemma below.
Note that our domain is rotated by $\pi/2$ with respect to the one 
of \cite[Proposition 2.3]{Darde-Ervedoza-2019}, 
but the adaptation is straightforward.
\begin{lemma} \label{Lem:Dardé-Ervedoza}
Denote $\Omega_1$ the domain $\Omega$ of \eqref{Def:Omega_and_square} 
obtained for $L=1$. There exists a unique function $\tilde{\varphi}$ satisfying
\begin{equation} \label{Eq:Sys-Phi}
\left\{ \begin{array}{l}
\Delta \tilde{\varphi} = -2\delta_{\mathcal{M}} \ \text{ in } \ \Omega_1, \\
\tilde{\varphi} = 0 \ \text{ on } \ \partial \Omega_1, \\
\displaystyle 
\lim_{|a| \rightarrow +\infty} \sup_{b \in (-|a|-1,|a|+1)} |\tilde{\varphi}(a,b)| =0.
\end{array} \right. 
\end{equation}
Besides, the maximum of $\tilde{\varphi}$ is attained in $0$:
\begin{equation} \label{Eq:MaxPhi}
\max_{\Omega_1} \widetilde{\varphi} = \tilde{\varphi}(0) =  
\frac{2}{\pi} \frac{\displaystyle\int_{0}^{\pi/2} 
    \ln\left(\cot\left(\frac{t}{2}\right)\right) \sqrt{\cos(t)} \, dt}
{\displaystyle \int_{0}^{\pi/2} \sqrt{\cos(t)} \, dt} 
= \frac{\Gamma(1/4)^2}{4\sqrt{2} \pi^2}\sum_{n\in \N} 
    \frac{(-1)^n}{2n+1}\frac{\Gamma(n+1/4)}{\Gamma(n+7/4)}.
\end{equation}
\end{lemma}
\noindent
Notice that due to the weak maximum principle, we have
\begin{equation} \label{Eq:TildePhiPos}
\tilde{\varphi} \geq 0 \text{ in } \Omega_1.
\end{equation} 
The constant given in \eqref{Eq:MaxPhi} has been recently put in simpler form
by Lissy \cite{lissy2026optimal}:
\begin{lemma}(\cite[Theorem 3.1]{lissy2026optimal})
One has the equality
\begin{equation*} 
\frac{1}{4} + \frac{\tilde{\varphi}(0)}{2} 
= \frac{\Gamma(1/4)^2}{4\pi\Gamma(3/4)^2}
= \frac{\Gamma(1/4)^4}{8\pi^3} = \kappa_\star.
\end{equation*}
\end{lemma}
\noindent
Now, following \cite[Section 3.2.2]{Darde-Ervedoza-2019}, we observe that 
\begin{equation} \nonumber 
(x,y) \in \Omega_1 \mapsto \tilde{\varphi}(x,y) + |x|,
\end{equation}
is harmonic in $\Omega_1$, and therefore equals the real part of a holomorphic function $\phi$:
\begin{equation} \label{Eq:PropPhi}
\phi \in \mathcal{H}(\Omega_1) \ \text{ such that } \Re(\phi(z)) = \tilde{\varphi}(z) + |\Re(z)|
\text{ in } \Omega_1.
\end{equation}
The following additional property of $\phi$ is proved in Subsection~\ref{Subsec:WeightHolderEstimate}.
\begin{lemma} \label{Lem:PhiHolder}
The function $\phi$ belongs to 
$C^{\infty}(\overline{\Omega_1} \setminus \{i,-i\}) \cap C^{2/3}_{loc}(\overline{\Omega_1})$ 
and satisfies
\begin{equation}
\label{Eq:PhiAuxCoins}
\nabla \phi(z) = \mathcal{O}(1/|z \pm i|^{1/3}) \text{ as } z \rightarrow \pm i 
\ \text{ in } \ \Omega_1.
\end{equation}
\end{lemma}
\begin{remark}
Lemma~\ref{Lem:PhiHolder} and Lemma~\ref{Lem:WeightHolderEstimate} below are
technical statements used to prove that one can attain exactly the constant 
$\kappa_\star$ in the exponential of Theorems~\ref{Thm:Main}, \ref{Thm:1prime} 
and \ref{Thm:DEbis}.
We could simplify the proof and discard them if we only wanted a cost of size
$\mathcal{O}(1) \exp\left( \kappa_\star \frac{\tilde{L}^2}{T} \right)$ for $\tilde{L}>L$.
\end{remark}
This function $\phi$ allows us to finally introduce
\begin{equation} \label{Eq:DefPsi}
\psi(z) = \exp\left( - \frac{z^2}{4T} - \frac{L^2}{4T}- \frac{L^2}{2T} \, 
\phi\left( \frac{z}{L} \right)\right) \text{ in } \Omega.
\end{equation}
%
%
%
%
\subsubsection{Estimating the $\A$ term.} 
Going back to estimating $\A$, we then let
\begin{equation} \nonumber 
\widetilde{\mathfrak{a}} (z) = \overline{\psi}(z) \A(z)
\text{ in } \Omega.
\end{equation}
Since $\phi$ is holomorphic in $\Omega_1$, 
we see that $\psi$ is holomorphic in $\Omega$, 
so that the multiplier $\overline{\psi}(z)$ is antiholomorphic. 
Recalling \eqref{Eq:MainEqA}, we infer that
\begin{multline*}
\partial \widetilde{\mathfrak{a}} 
= \widetilde{\mathcal{B}}
\text{ with } 
\tilde{\mathcal{B}}(z) :=
- \exp\left(\overline{\varphi}_1(z) -\frac{\overline{z}^2}{4T} - \frac{L^2}{4T} - \frac{L^2}{2T} \overline{\phi}\left(\frac{z}{L}\right) \right) 
    \widehat{B}_1 [e^{-2\lambda - \overline{\varphi}_1} \widehat{v}](z) \\
 + \exp\left( \overline{\varphi}_2(z) -\frac{\overline{z}^2}{4T} - \frac{L^2}{4T} - \frac{L^2}{2T} \overline{\phi}\left(\frac{z}{L}\right) \right) 
    \widehat{B}_2 [e^{-2\lambda - \overline{\varphi}_2}\widehat{v}](z).
\end{multline*}
Using \eqref{Def:Phi-i} and then \eqref{Eq:PropPhi}, we find that for $k=1,2$,
\begin{align} \nonumber 
\exp\left(\overline{\varphi}_k -\frac{\overline{z}^2}{4T} - \frac{L^2}{4T} 
\right. &\left. - \frac{L^2}{2T} \overline{\phi}\left(\frac{z}{L}\right)\right) \\
\nonumber 
&= \exp\left( (-1)^{k+1} \frac{L\overline{z}}{2T}  - \frac{L^2}{2T} \overline{\phi}\left(\frac{z}{L}\right)\right)  \\
\label{Eq:ExpressionPsiPhi}
&= \exp\left( \frac{L((-1)^{k+1} \overline{z}-|\Re(z)|)}{2T}  
    - \frac{L^2}{2T} \tilde{\varphi}\left(\frac{z}{L}\right)
    + i\frac{L^2}{2T} \Im{\phi}\left(\frac{z}{L}\right) \right)  .
\end{align}
Now we use \eqref{Eq:TildePhiPos} to deduce
that for $k=1,2$,
\begin{equation} \label{Eq:MultiplicateurBorne}
\left| \exp\left(\overline{\varphi}_k -\frac{\overline{z}^2}{4T} - \frac{L^2}{4T} - \frac{L^2}{2T} \overline{\phi}\left(\frac{z}{L}\right)\right) \right|
\leq  \exp\left( - \frac{L^2}{2T}\tilde{\varphi}\left(\frac{z}{L}\right)\right) \leq 1.
\end{equation}
It follows that
\begin{equation*}
\|\tilde{\mathcal{B}}(z)\|_{L^2(\Omega)} \leq 
\| B_1 [e^{-2\lambda - \overline{\varphi}_1} \widehat{v}]\|_{L^2(D_1)}
+
\| B_2 [e^{-2\lambda - \overline{\varphi}_2} \widehat{v}]\|_{L^2(D_2)}.
\end{equation*}
Another consequence of \eqref{Eq:MultiplicateurBorne} is the following.
\begin{lemma} \label{Lem:afrakinL2}
One has $\tilde{\mathfrak{a}} \in H^1_0(\Omega)$.
\end{lemma}
\begin{proof}[Proof of Lemma~\ref{Lem:afrakinL2}]
From Lemma~\ref{Lem:Est-Asymp-alphai} we deduce that $\alpha_1$ and $\alpha_2$
belong to $L^2(\Omega)$, so \eqref{Eq:MultiplicateurBorne} gives 
$\tilde{\mathfrak{a}} \in L^2(\Omega)$. 
Now consider the derivatives
\begin{equation*}
\partial_k (\overline{\psi} \mathfrak{a}) =
\partial_k (\overline{\psi} e^{\overline{\varphi_1}}) \widehat{\alpha}_1 
+ (\overline{\psi} e^{\overline{\varphi_1}}) \partial_k \widehat{\alpha}_1 
-\partial_k (\overline{\psi} e^{\overline{\varphi_2}}) \widehat{\alpha}_2 
- (\overline{\psi} e^{\overline{\varphi_2}}) \partial_k \widehat{\alpha}_2 .
\end{equation*}
The second and fourth term in the right-hand side are clearly in $L^2(\Omega)$ due to
\eqref{Eq:MultiplicateurBorne}. Consider the first one, the third one being similar.
We note that $\nabla\tilde{\varphi}(z) \rightarrow 0$ as $|z| \rightarrow +\infty$ inside $\Omega$
due to \eqref{Eq:Sys-Phi} and local elliptic estimates. 
With \eqref{Eq:PropPhi}, we deduce that $|\nabla \Re(\phi)| \lesssim 1$ outside a bounded set.
Hence the same can be said of $\Im(\phi)$ by the Cauchy-Riemann equation, since it is the
harmonic conjugate of $\Re(\phi)$.
With \eqref{Eq:ExpressionPsiPhi}, this involves that 
$\partial_k (\overline{\psi} e^{\overline{\varphi_1}})$ is bounded outside of a bounded set. \par
On the rest of $\Omega$, the only singularities are at $\pm iL$, due to the singularity
of $\nabla \phi$ there. 
However due to \eqref{Eq:PhiAuxCoins}, one has 
$\partial_k (\overline{\psi} e^{\overline{\varphi_1}}) \in L^4$ locally at $\pm iL$, 
and so does $\widehat{\alpha}_1$ because it is in $H^1_{loc}$. 
This proves $\nabla \mathfrak{a} \in L^2(\Omega)$. \par
Finally, the vanishing trace of $\mathfrak{a}$ on the boundary comes from 
$\widehat{\alpha}_k \in \dot{H}^1_0(\Omega)$, $k=1,2$.
\end{proof}
Now due to Lemma~\ref{Lem:Poisson-Omega}, there exists a unique 
$\beta \in \dot{H}^1_0(\Omega)$ solution to 
\begin{equation*}
\Delta \beta = 4 \overline{\partial} \tilde{\mathcal{B}},
\end{equation*}
which moreover satisfies $\| \nabla \beta \|_{L^2(\Omega)} \lesssim \| \tilde{\mathcal{B}}\|_{L^2(\Omega)}$. 
Due to \eqref{Eq:Laplacian}, $\tilde{\mathfrak{a}} - \beta$ is harmonic, 
and moreover is null on $\partial \Omega$. 
The uniqueness part in Lemma~\ref{Lem:Poisson-Omega} involves that $\tilde{\mathfrak{a}} = \beta$. We deduce the main estimate on $\A$:
\begin{equation} \label{Eq:MainEstimateA}
\| \nabla \tilde{\mathfrak{a}} \|_{L^2(\Omega)} \lesssim
\left( \|B_1  \circ E_1 [e^{-2\lambda - \overline{\varphi}_1} \, {v}]\|_{L^2(D_1)}
+
\|B_2  \circ E_2 [e^{-2\lambda - \overline{\varphi}_2} \, {v}]\|_{L^2(D_2)} 
\right).
\end{equation}
%
%
%
%
%
%
%
\subsubsection{Proof of the main estimate \eqref{Eq:Central-Goal}.}
This estimate will rely on \eqref{Eq:Decomp-v-i} and the Cauchy-Green formula \eqref{Eq:Cauchy-Green2} for the $\partial$ operator,
used separately on $\sql$ and $\sqr$ (recall the notation from \eqref{Eq:Subdomains}). \par
\ \par
\noindent
{\it a. On $\sql$.}
We start from
\begin{eqnarray} \nonumber
\int_{\sql} e^{-2\lambda} |v|^2 \, dA(z) 
&=& \int_{\sql} \left( e^{-2\lambda - \overline{\varphi}_1} v \right)
    e^{\overline{\varphi}_1} \overline{v} \, dA(z) \\
\label{Eq:ITerm}
&=& \int_{\sql} B_1 \left( e^{-2\lambda - \overline{\varphi}_1} v \right)
    e^{\overline{\varphi}_1} \overline{v} \, dA(z) \\
\label{Eq:IITerm}
&\ & + \int_{\sql} \left( \partial \alpha_1 \right)
    e^{\overline{\varphi}_1} \overline{v} \, dA(z) .
\end{eqnarray}
We call $I$ and $\II$ the terms in \eqref{Eq:ITerm} and \eqref{Eq:IITerm}, respectively. \par
\ \par
\noindent
$\bullet$ Concerning the term $I$, we write
\begin{equation*}
I = \int_{\sql} e^{-\lambda}\overline{v} \, 
    \left(e^{\overline{\varphi}_1} e^{\lambda}\right) \,
    B_1 \left( e^{-2\lambda - \overline{\varphi}_1} v \right) \, dA(z).
\end{equation*}
We notice that for $z=x+iy$
\begin{equation*}
\left|e^{\overline{\varphi}_1(z)} e^{\lambda(z)} \right|
= e^{\frac{(x+L)^2}{4T}}
\leq \exp\left( \frac{L^2}{4T} \right)
\ \text{ on } \ \sql,
\end{equation*}
so using the Cauchy-Schwarz inequality we find
\begin{equation} \nonumber 
|I| \leq \exp\left( \frac{L^2}{4T} \right) \, 
\left\| B_1 \left( e^{-2\lambda - \overline{\varphi}_1} v \right) \right\|_{L^2(D_1)} 
\|v\|_{\Hsq}.
\end{equation}
\ \par
\noindent
$\bullet$ Concerning the term $\II$, we first notice that $e^{\varphi_1}\, v \in \mathscr{H}(\sq)$,
so $\partial \left( \overline{e^{\varphi_1}\, v} \right)=0$ in $\sq$.
Using \eqref{Eq:Cauchy-Green2} and $\alpha_1 \in \dot{H}^1_0(D_1)$, we find
\begin{equation*}
\II 
= - \frac{1}{2i} \int_{\partial \sql} \alpha_1 \, e^{\overline{\varphi}_1} \, 
    \overline{v} \, d \overline{z} 
= - \frac{1}{2i} \int_{\SN} \alpha_1 \, e^{\overline{\varphi}_1} \, 
    \overline{v} \, d \overline{z},
\end{equation*}
where $\SN$ was introduced in Paragraph~\ref{Subsubsec:Domains}
and where we used the regularity of $v$ supposed in the beginning;
the integral could be replaced with the
$\left\langle \cdot, \, \cdot 
    \right\rangle_{H^{1/2}(\partial \sql) \times H^{-1/2}(\partial \sql)}$ 
duality product (this is further discussed below).
\par
\ \par
\noindent
{\it b. On $\sqr$.}
We make the same computations on $\sqr$, {\it mutatis mutandis}, and find that
\begin{equation} \label{Eq:RightSide}
\int_{\sqr} e^{-2\lambda} |v|^2 \, dA(z) 
= I' - \frac{1}{2i}  \int_{\NS} \alpha_2 \, e^{\overline{\varphi}_2} \, 
    \overline{v} \, d \overline{z},
\end{equation}
with
\begin{equation} \label{Est-I-Prime}
|I'| \leq \exp\left( \frac{L^2}{4T} \right) \, 
\left\| B_2 \left( e^{-2\lambda - \overline{\varphi}_2} v \right) \right\|_{L^2(D_2)} 
\, \|v\|_{\Hsq}.
\end{equation}
Of particular importance is the fact that the line integral in \eqref{Eq:RightSide}
is from $iL$ to $-iL$,  due to the trigonometric orientation of $\partial \sqr$ 
in this computation. \par
\ \par
\noindent
{\it c. Gathering the two sides.}
Summing the previous estimates on $\sql$ and $\sqr$, we find that
\begin{equation} \label{Eq:Cutting-Hsquare-norm}
\int_{\sq} e^{-2\lambda} |v|^2 \, dA(z) = I'' + \II'',
\end{equation}
where
\begin{equation*}
\left| I'' \right| \leq \exp\left( \frac{L^2}{4T} \right) \, 
\, \|v\|_{\Hsq}
\left( 
\left\| B_1 \left( e^{-2\lambda - \overline{\varphi}_1} v \right) \right\|_{L^2(D_1)}  
+ \left\| B_2 \left( e^{-2\lambda - \overline{\varphi}_2} v \right) \right\|_{L^2(D_2)}  
\right),
\end{equation*}
and
\begin{equation}
\label{Eq:Form-II'}
\II'' 
= - \frac{1}{2i} \int_{\SN} \left( 
e^{\overline{\varphi}_1} \, \alpha_1  - e^{\overline{\varphi}_2} \, \alpha_2 
\right) \overline{v} \, d \overline{z} 
= - \frac{1}{2i} \int_{\SN} \A\, \overline{v} \, d \overline{z}.
\end{equation}
Now the formal computation would be to use the duality inequality
\begin{equation*}
|\II''|  \leq \, \| e^{\lambda} \A \|_{H^{1/2}(\SN)} 
\|e^{-\lambda} v\|_{H^{-1/2}(\SN)},
\end{equation*}
and to treat $\|e^{-\lambda} v\|_{H^{-1/2}(\SN)}$ with Lemma~\ref{Lem:Traces-Bergman}. 
However there are two difficulties:
\begin{itemize}
\item[--] the restriction operator from $H^{-1/2}(\partial\sql)$ to $H^{-1/2}(\mathcal{M})$ is not 
well-defined (see e.g. \cite[Chapter 33]{tartar2007introduction}). 
The answer to this technical issue is to notice that
$\A_{|\SN}$ belongs to the Lions-Magenes space $H^{1/2}_{00}(\SN)$
(the subspace of $H^{1/2}(\SN)$ made of functions that once extended by $0$ 
outside $\SN$ belong to $H^{1/2}(i\R)$, cf. ibid and 
\cite[Chapter 11]{lions2012non}). 
Equivalently, we extend $\A_{|\mathcal{M}}$ 
by $0$ in $\partial\sql$ and work in this domain.
\item[--] we have estimates on $\tilde{\A}$ rather than on $\A$; so we will actually
decompose $\A\, \overline{v}$ as 
\begin{equation*}
\A\, \overline{v} 
= (e^\lambda \overline{\psi}^{-1})\cdot (\overline{\psi} \A) \cdot (e^{-\lambda} v),
\end{equation*}
that we will estimate in $C^{2/3} \times H^{1/2}_{00} \times H^{-1/2}$.
\end{itemize}
\ \par
\noindent
$\bullet$ Concerning $\|e^{-\lambda} v\|_{H^{-1/2}(\sql)}$,
we use Lemma~\ref{Lem:Traces-Bergman} to deduce
\begin{equation*}
\|e^{-\lambda} v\|_{H^{-1/2}(\partial\sql)}
\lesssim
\left( \|e^{-\lambda} v\|_{L^{2}(\sql)} 
    + \left\|\overline{\partial} \left( e^{-\lambda} v \right)
        \right\|_{L^{2}(\sql)} \right).
\end{equation*}
Since 
$\overline{\partial} \left( e^{-\lambda} v \right) 
    = -(\overline{\partial} \lambda) \left( e^{-\lambda} v \right)$,
we find
\begin{equation} \label{Eq:EstH-1/2}
\|e^{-\lambda} v\|_{H^{-1/2}(\partial\sql)} \lesssim \frac{1}{T} \| v\|_{\Hsq}.
\end{equation}
$\bullet$ We now consider $e^{\lambda} \overline{\psi}^{-1}$.
For $z =x+iy \in \Omega$, recalling \eqref{Eq:DefPsi}, we have
\begin{equation*}
e^{\lambda} \overline{\psi}^{-1} 
= \exp\left( \frac{y^2}{4T} + \frac{\overline{z}^2}{4T} + \frac{L^2}{4T}
+ \frac{L^2}{2T} \, \overline{\phi}\left( \frac{z}{L} \right)\right).
\end{equation*}
On $\SN$, we have $\overline{z}^2 = - \Im(z)^2$ so
\begin{equation} \label{Eq:ProdPoidsMult}
e^{\lambda} \overline{\psi}^{-1} 
= \exp\left( \frac{L^2}{4T}
+ \frac{L^2}{2T} \, \overline{\phi}\left( \frac{z}{L} \right)\right).
\end{equation}
With \eqref{Eq:PropPhi}, since $\Re(z)=0$ on $\SN$, we get:
\begin{equation} \label{Eq:ApparitionKappaStar}
\left| e^{\lambda} \overline{\psi}^{-1} \right|  
= \exp\left( \frac{L^2}{4T}+ \frac{L^2}{2T} \, 
\tilde{\varphi}\left( \frac{z}{L} \right)\right)
\text{ for all } z \in \SN.
\end{equation}
The maximum of this real-valued function on $\SN$ is obtained at $z=0$ and 
gives 
\begin{equation*}
\exp\left( \frac{L^2}{T}\frac{1 + 2 \tilde{\varphi}(0)}{4} \right)
= \exp\left( \kappa_\star \frac{L^2}{T} \right),
\end{equation*}
so
\begin{equation} \label{Eq:MainEstMultiplier}
\| e^{\lambda} \overline{\psi}^{-1} \|_{\infty} \leq \exp\left( \kappa_\star \frac{L^2}{T} \right).
\end{equation}
Actually, we will need a more regular estimate on $e^{\lambda} \overline{\psi}^{-1}$. 
Precisely the following estimate is established in 
Subsection~\ref{Subsec:WeightHolderEstimate}.
\begin{lemma} \label{Lem:WeightHolderEstimate}
The function $e^{\lambda} \overline{\psi}^{-1}$ restricted to $\mathcal{M}$
is of H\"older class $C^{2/3}(\mathcal{M})$ and satisfies
\begin{equation}
\label{Eq:WeightHolderEstimate}
| e^{\lambda} \overline{\psi}^{-1} |_{C^{2/3}(\mathcal{M})} \lesssim \frac{1}{T} 
\exp\left( \kappa_\star \frac{L^2}{T} \right).
\end{equation}
\end{lemma}
\noindent
$\bullet$
Let us now focus on $\overline{\psi} {\A}=\tilde{\A}$. 
Since, due to Lemma~\ref{Lem:afrakinL2}, we have 
$\tilde{\A} \in H^1_0(\Omega)$,
we can extend it to $\C$ by $0$. Call $\check{\A}$ this extension; 
clearly it has a trace in $H^{1/2}(i\R)$.
Now we call $\widehat{\A}$ the function on $\partial \sql$ obtained by extending
$\tilde{\A}$ by $0$ on $\partial \sql \setminus \mathcal{M}$.
This gives an $H^{1/2}$ function on $\partial \sql$, because it is defined by means 
of local maps based on the extension of $\check{\A}$. 
Moreover, $\widehat{\A}$ has an $H^{1/2}$-norm equivalent to the one of 
$\check{\A}$. Using Poincaré's inequality, we infer that 
\begin{equation} \label{Eq:EstCheckA}
\| \widehat{\A}\|_{H^{1/2}(\partial \sql)} 
    \lesssim \|\tilde{\A}\|_{H^1(\Omega \cap \{|\Re(z)| <1\})}
    \lesssim \|\nabla \tilde{\A}\|_{L^2(\Omega)}.
\end{equation} \par
\ \par
Now we can resume the computation left in \eqref{Eq:Form-II'}:
\begin{align*}
|\II''| 
&= \frac{1}{2} \left| \int_{\SN} \A\, \overline{v} \, d \overline{z} \right| \\
&= \frac{1}{2} \left| \int_{\sql} \widehat{\A} \cdot \Pi_{\sql}(e^\lambda \overline{\psi}^{-1})
    \cdot (e^{-\lambda}\overline{v}) \, d \overline{z} \right| \\
&\lesssim \|\widehat{\A}\|_{H^{1/2}(\partial\sql)} 
          \|\Pi_{\sql}(e^\lambda \overline{\psi}^{-1})\|_{C^{2/3}(\partial\sql)}
          \|e^{-\lambda}\overline{v}\|_{H^{-1/2}(\partial \sql)},
\end{align*}
where $\Pi_{\sql}$ is a linear continuous operator 
$C^{2/3}(\mathcal{M}) \rightarrow C^{2/3}(\partial \sql)$
and where we used that $C^{2/3}(\partial \sql)$ is a multiplier space 
on $H^{1/2}(\partial \sql)$.
Using \eqref{Eq:EstH-1/2}, \eqref{Eq:MainEstMultiplier}-\eqref{Eq:WeightHolderEstimate}
and \eqref{Eq:EstCheckA}, we deduce
\begin{equation*}
|\II''|
\lesssim \frac{1}{T^2}e^{\kappa_\star \frac{L^2}{T}} \,
    \| \nabla \tilde{\A} \|_{L^{2}(\Omega)} \| v\|_{\Hsq}.
\end{equation*}
We inject the above estimates on $I''$ and $\II''$ in \eqref{Eq:Cutting-Hsquare-norm}.
Simplifying by $\| v\|_{\Hsq}$ and using $\kappa_\star \geq \frac{1}{4}$, we get:
\begin{equation*}
\| v\|_{\Hsq} \lesssim \frac{1}{T^2}
\exp\left( \frac{\kappa_\star L^2}{T} \right) \| \nabla \tilde{\A} \|_{L^{2}(\Omega)}.
\end{equation*}
Using \eqref{Eq:MainEstimateA}, we find
\begin{equation*}
\| v\|_{\Hsq} \lesssim \frac{1}{T^2}
\exp\left( \frac{\kappa_\star L^2}{T} \right) \left[ 
\left\| B_1 \circ E_1 \left( e^{-2\lambda - \overline{\varphi}_1} v \right) \right\|_{L^2(D_1)}  
+ \left\| B_2 \circ E_2 \left( e^{-2\lambda - \overline{\varphi}_2} v \right) \right\|_{L^2(D_2)} 
\right].
\end{equation*}
This proves \eqref{Eq:Observability}.
\qed
\begin{remark}
Following the same lines with $\varphi_1 = \varphi_2 = \lambda= 0$ and without the need 
of the multiplier $\psi$ would give an alternative proof 
to Hartmann \& Orsoni's result that  $L^2_a(\sq) \subset L^2_a(D_1) + L^2_a(D_2)$, 
see \cite[Corollary 1.6]{Hartmann-Orsoni}.
Their full result is however more general, considering $L^p$ spaces and covering more
diverse geometric situations.
\end{remark}
%
%
%
%
%
%
%
%
%
\section{Proofs concerning the Poisson equation}
\label{Sec:Technical-proofs}
%
%
%
%
%
%
\subsection{Poisson equation in the quadrant: proofs of Lemmas~\ref{Lem:Poisson}, 
\ref{Lem:Bergman-on-Di} and \ref{Lem:Est-Asymp-alphai}}
\label{Subsec:Proof-Poisson-D}
\begin{proof}[Proof of Lemma~\ref{Lem:Poisson}]
This statement is more or less classical, but since the domain is unbounded and
merely Lipschitz, some explanations are in order.
We follow Sohr \& Specovius-Neugebauer \cite{SohrSpecovius1998} and 
Ortner \& Süli \cite{ortner2012note}.
The first part of the statement can be obtained in a standard way by means of the
Lax-Milgram approach. First, we have the following.
\begin{lemma} \label{Lem:dotH10}
The space $\dot{H}^1_0(Q)$ is a Hilbert space when equipped with the scalar product
\begin{equation*}
\left\langle u, \, v \right\rangle_{\dot{H}^1_0(Q)} 
= \int_{Q} \nabla u \cdot \nabla v \, dx.
\end{equation*}
Moreover, $C^{\infty}_c(Q)$ is dense in $\dot{H}^1_0(Q)$.
\end{lemma}
\begin{proof}
The fact that $\left\langle \cdot, \, \cdot \right\rangle_{\dot{H}^1_0(Q)}$ is a scalar
product is clear, due to the zero trace on the boundary, and the completeness is
straightforward, since for a Cauchy sequence $\left( u_n \right)_{n\in \N}$, the
sequence $(\nabla u_n)$ converges in $L^2(Q)$ to a distributional gradient belonging
to $L^2(Q)$, and one can use local trace inequalities to pass to the limit on the 
boundary. \par
To prove the density of $C^{\infty}_c(Q)$ in $\dot{H}^1_0(Q)$, we use a symmetrization
of the functions with respect to the axes. Precisely we call
\begin{multline*}
Q_1 := Q, \ \ 
Q_2:= \left\{ z \in \C \ \Big/ \ \Re(z) \leq 0, \ \Im(z) \geq 0 \right\}, \\
Q_3:= \left\{ z \in \C \ \Big/ \ \Re(z) \leq 0, \ \Im(z) \leq 0 \right\}
\ \text{ and } \ 
Q_4:= \left\{ z \in \C \ \Big/ \ \Re(z) \geq 0, \ \Im(z) \leq 0 \right\},
\end{multline*}
and to $f:Q_1 \rightarrow \C$, we associate $\mathfrak{S}_1[f]$ by 
\begin{equation} \label{Eq:Symmetrizer}
\mathfrak{S}_1[f](z):=
\left\{ \begin{array}{ll}
f(z) & \ \text{ in } \ Q_1, \\
-{f}(-\overline{z}) & \ \text{ in } \ Q_2, \\
 f(-z) & \ \text{ in } \ Q_3, \\
-{f}(\overline{z}) & \ \text{ in } \ Q_4,
\end{array} \right. 
\end{equation}
Correspondingly, we can define symmetrizers $\mathfrak{S}_k$, $k=2,3,4$,
that extend functions $f:Q_k \rightarrow \C$ to functions
$\C \rightarrow \C$ which are odd with respect to both real and imaginary parts. \par
Now we use the density of $C^{\infty}_c(\R^2) / \R$ in 
\begin{equation*} 
\dot{H}^1(\R^2) : = \left\{ u \in L^1_{loc}(\R^2) \ / \ \nabla u \in L^2(\R^2) \right\} 
    \Big/ \R,
\end{equation*}
by which we mean the equivalence classes of functions up to an additive constant.
See e.g. \cite[Theorem 1]{SohrSpecovius1998} and \cite[Theorem 2.1]{ortner2012note}.
Note that this space should not be confused with the homogeneous Sobolev space
defined via the Fourier transform as the space of $L^1_{loc}(\R^2)$ functions $u$
satisfying $\int_{\R^2} |\xi|^2 | \mathscr{F} u (\xi)|^2\, d\xi < +\infty$; 
this latter space does not form a Hilbert space, see \cite[Proposition 1.34]{MR2768550}. \par
Then given $u \in \dot{H}^1_0(Q)$, we find a sequence 
$(v_n) \in C^{\infty}_c(\R^2)^{\N}$ such that
\begin{equation*}
v_n \longrightarrow \mathfrak{S}_1(u) \text{ in } \dot{H}^1(\R^2),
\text{ that is, }
\nabla v_n \longrightarrow \nabla \mathfrak{S}_1(u) \text{ in } L^2(\R^2).
\end{equation*}
We then consider the symmetrized version:
\begin{equation*}
\tilde{v}_n := \frac{\mathfrak{S}_1 \left( v_{n|Q_1} \right)
+\mathfrak{S}_2 \left( v_{n|Q_2} \right)
+\mathfrak{S}_3 \left( v_{n|Q_3} \right)
+\mathfrak{S}_4 \left( v_{n|Q_4} \right) }{4}.
\end{equation*}
Note that for all $z$ in $\C$,
\begin{equation*}
\tilde{v}_n(z) = \frac{1}{4}\left[ v_n(z) + v_n(-z) 
-v_n(\overline{z}) -v_n(-\overline{z}) \right].
\end{equation*}
In other words, one can see $\tilde{v}_n$ as the projection of $v_n$ 
on functions that are odd with respect to both $\Re(z)$ and $\Im(z)$. \par
Then the sequence $(\tilde{v}_{n})$ belongs to $C^\infty_c(\R^2)^{\N}$
and satisfies  $\tilde{v}_n = \mathfrak{S}_1(\tilde{v}_{n|Q_1})$.
In particular its restriction to $Q$ belongs to $\dot{H}^1_0(Q)$ and 
$\tilde{v}_{n|Q} \rightarrow u$ in $\dot{H}^1_0(Q)$. \par
It remains to approximate $\tilde{v}_{n|Q}$ by a function with compact
support in $Q$ rather than $\overline{Q}$.
It is mainly a matter of considering 
\begin{equation*}
\rho\left( \frac{x}{\varepsilon} \right) 
    \rho\left( \frac{y}{\varepsilon'} \right) v(x,y),
\end{equation*}
with $\rho \in C^{\infty}(\R;[0,1])$ with $\rho \equiv 1$ on $\R \setminus (-1,1)$.
Let us explain why $\rho\left( \frac{y}{\varepsilon'} \right) v(x,y)$ converges to $v$
in $\dot{H}^1_0(Q)$ as $\varepsilon' \rightarrow 0$. Call $K_n := \supp(v_n)$.
Then 
\begin{equation*}
\nabla \left\{ [1-\rho]\left( \frac{y}{\varepsilon'} \right)  v_n(x,y)  \right\}
= [1-\rho]\left( \frac{y}{\varepsilon'} \right) \nabla v_n(x,y)  
- \frac{1}{\varepsilon'} v_n(x,y) \partial_y \rho\left( \frac{y}{\varepsilon'} \right)
\end{equation*}
The first term converges to $0$ in $L^2(Q)$ as $\varepsilon' \rightarrow 0$
due to Lebesgue's dominated convergence theorem. 
For the second one, we have $v_n(x,y) = \mathcal{O}(\varepsilon')$ uniformly
in $K_n \cap \supp \rho\left( \frac{\cdot}{\varepsilon'}\right)$, and the
measure of this latter set is $\mathcal{O}(\varepsilon')$.
Taking care of the $\rho\left( \frac{x}{\varepsilon} \right)$ factor is similar.
\end{proof}
%
%
%
\noindent
{\it Back to the proof of Lemma~\ref{Lem:Poisson}}
Now thanks to the density mentioned above, Equation~\eqref{Eq:PoissonQ} 
in the sense of distributions is equivalent to 
\begin{equation*}
\forall v \in \dot{H}^1_0(Q), \ \ 
\left\langle u, \, v \right\rangle_{\dot{H}^1_0(Q)} 
= \left\langle F, \, \nabla v \right\rangle_{L^2(Q)}.
\end{equation*}
The existence and uniqueness of such a $u$ given $F$ is then a direct consequence
of the Lax-Milgram theorem or the Riesz representation theorem, which also gives
\eqref{Eq:Est-Poisson-0} at the same time, taking $v=u$. 
Then \eqref{Eq:Est-Poisson-1} is just a consequence of a local Poincaré inequality.
\end{proof}
\ \par
%
%
%
%
%
%
%
\begin{proof}[Proof of Lemma~\ref{Lem:Bergman-on-Di}] \par
We let
\begin{equation*}
\alpha := 4 \Delta^{-1}_{Q}[\overline{\partial} v] .
\end{equation*}
This is well-defined thanks to the previous lemma, with the obvious observation
that $\overline{\partial} v$ can be put in the form $\div F$ by considering
real and imaginary parts.
Then to prove \eqref{Eq:Bergman-on-Di}, it suffices to prove that 
$\partial \alpha \perp L^2_a(Q)$ and that $v - \partial \alpha \in L^2_a(Q)$.
The first part is due to the following Cauchy-Green formula 
(which is valid in the context of $\dot{H}^1_0(Q)$ due to the density of 
$C^{\infty}_{c}(Q)$ in $\dot{H}^1_0(Q)$): for $w \in L^2_a(Q)$,
\begin{equation*}
\int_{Q} w \, \overline{\partial \alpha}\, dA(z) 
= - \left\langle \overline{\partial} w , \, \alpha 
    \right\rangle_{H^{-1} \times \dot{H}^1_0(Q)} 
    + \left\langle  w, \, \alpha 
        \right\rangle_{H^{-1/2}(\partial Q) \times H^{1/2}(\partial Q)}.
\end{equation*}
Then the terms in the right-hand side vanish. \par
For the second part, using \eqref{Eq:Laplacian}, we see that
\begin{equation*}
\overline{\partial} \left( v - \partial \alpha \right) = 0.
\end{equation*}
Hence $v - \partial \alpha \in L^2_a(Q)$.
\end{proof}
\ \par
%
%
%
\begin{proof}[Proof of Lemma~\ref{Lem:Est-Asymp-alphai}]
It is enough to treat the case $k=1$. Due to the support of $v$, 
$\overline{\partial} \, E_1[e^{-2\lambda - \overline{\varphi_1}} v]$ is supported in
$\overline{\sq}$, and consequently $\alpha_1:D_1 \rightarrow \C$ is harmonic
in $D_1 \setminus D_2$. 
Now we define $\check{\alpha}_1$ as $\check{\alpha}_1 = \alpha_1 \circ T$, 
where $T$ is the unique rigid movement (translation-rotation) sending $Q$ on $D_1$.
Symmetrizing $\check{\alpha}_1$ as in \eqref{Eq:Symmetrizer}, 
we obtain a function $\tilde{\alpha}_1$ that is harmonic 
outside of the square of diagonal $[-(1+i)\sqrt{2} L, (1+i)\sqrt{2}L]$
(as a consequence of Schwarz's reflection principle for harmonic functions). 
So $\partial \tilde{\alpha}_1$ is holomorphic outside of this square. 
Since it is also in $L^2(\C)$, its Laurent series decomposition has only powers that
are less than or equal to $-2$:
\begin{equation*}
\partial \tilde{\alpha}_1(z) = \sum_{k=2}^{+\infty} c_k z^{-k} 
\text{ for } |z| > 3L.
\end{equation*}
Moreover, by construction, $\tilde{\alpha}_1$ is even (for all $z \in \C$, 
$\tilde{\alpha}_1(-z) = \tilde{\alpha}_1(z)$), and consequently $\partial \tilde{\alpha}_1$
is odd. It follows that one also has $c_2=0.$ \par
Thus, for $|z|>4L$,
\begin{equation*}
\left| \partial \tilde{\alpha}_1(z) \right| 
    \leq \frac{K}{|z|^3},
\end{equation*}
for some constant $K$ depending on $L$, $T$ and $v$. Of course, the same estimate consequently applies to $\partial \alpha_1$. \par
The same reasoning can be done with $\overline{\alpha}_1$, which leads to the same estimate on 
$\partial \overline{\alpha}_1 = \overline{\overline{\partial} \alpha}_1$. 
Merging the estimate on $\partial \alpha_1$ and $\overline{\partial} \alpha_1$
one gets an estimate (enlarging $K$ a bit):
\begin{equation*}
\left| \nabla \alpha_1(z) \right| 
    \leq \frac{K}{|z|^3}.
\end{equation*}
Then \eqref{Eq:Est-Asymp-alphai} follows by line integration from 
a point of $\partial Q$ to $z$.
\end{proof}
%
%
%
%
%
%
%
%
%
%
\subsection{Poisson equation in $\Omega$: proof of Lemma~\ref{Lem:Poisson-Omega}}
\label{Subsec:Proof-Poisson-Omega}

Proceeding as for Lemma~\ref{Lem:Poisson}, it suffices to prove:
\begin{lemma} \label{Lem:dotH10Omega}
The space $\dot{H}^1_0(\Omega)$ is a Hilbert space when equipped with the scalar product
\begin{equation*}
\left\langle u, \, v \right\rangle_{\dot{H}^1_0(\Omega)} 
= \int_{\Omega} \nabla u \cdot \nabla v \, dx.
\end{equation*}
Moreover, $C^{\infty}_c(\Omega)$ is dense in $\dot{H}^1_0(\Omega)$.
\end{lemma}
\begin{proof}[Proof of Lemma~\ref{Lem:dotH10Omega}]
This is a light argument of domain decomposition.
We let $\rho \in C^{\infty}(\overline{\C};\, [0,1])$ such that
\begin{equation*}
\rho = 0 \text{ on } \overline{B}(0;\, 2L)
\ \text{ and } \ 
\rho = 1 \text{ on } \C \setminus B(0;\, 3L).
\end{equation*}
We also introduce $\rho_k \in C^{\infty}(\overline{\C};\, [0,1])$, $k=1,2$,
\begin{equation*}
\rho_k := \rho \, {\bf 1}_{D_k} 
\text{ so that } \rho =\rho_1 + \rho_2.
\end{equation*}
Given $u$ in $\dot{H}^1_0(\Omega)$, we consider $\rho_1 u$, $\rho_2 u$ and $(1-\rho) u$.
The functions $\rho_1 u$ and $\rho_2 u$ can be arbitrarily approximated by
functions in $C^{\infty}_{c}(Q_1)$ and $C^{\infty}_{c}(Q_2)$, respectively, due
to Lemma~\ref{Lem:dotH10}. 
The function $(1-\rho)u$ is supported in the bounded Lipschitz domain 
$\mathfrak{O} := \Omega \cap B(0,3L)$ and belongs to $H^1_0(\mathfrak{O})$.
That it can be approximated by a function in $C^{\infty}_c(\mathfrak{O})$ is then a consequence
of \cite[Corollary 1.5.1.6]{grisvard1985elliptic}.
\end{proof}
Then Lemma~\ref{Lem:Poisson-Omega} follows as before as a consequence of the 
Lax-Milgram theorem. \qed
%
%
%
%
%
%
%
%
%
%
%
%
\subsection{Regularity of the multiplier: proofs of Lemmas~\ref{Lem:PhiHolder} 
and \ref{Lem:WeightHolderEstimate}}
\label{Subsec:WeightHolderEstimate}
\begin{proof}[Proof of Lemma~\ref{Lem:PhiHolder}]
The main problem is the regularity of $\phi$ near the angles of $\partial \Omega_1$. \par
Since obviously $\phi$ is smooth inside $\Omega_1$, we only need to prove the regularity
up to the upper and lower boundaries of $\Omega_1$.
By symmetry we only treat the upper case. 
Call $\partial \Omega_{1}^+ := \partial \Omega_1 \cap \{\Im(z)>0\}$ 
the upper boundary of $\Omega_1$.
Due to \eqref{Eq:PropPhi} and the definition of $\Omega_1$, 
one has $\Re(\phi(z)) = |\Re(z)| = \Im(z)-1$ on $\partial \Omega_{1}^+$. Let 
\begin{equation*}
\phi_r(z) := \Re(\phi(z)) - \Im(z)+1
\text{ in } \Omega_1.
\end{equation*}
Then $\phi_r$ is harmonic in $\Omega_1$, and is equal to $0$ on $\partial \Omega_{1}^+$
(in the sense of $H^1_{loc}$ traces).
Now we consider the conformal mapping $\mathcal{G}$ from 
$\C_+ :=\left\{ \Re(z)>0 \right\}$ to $\{ \Im(z) < |\Re(z)| + 1\}$ 
given by $z \mapsto i-iz^{3/2}$. 
Then $\phi_r \circ \mathcal{G}$ is harmonic in some neighborhood of $i\R$ 
in $\overline{\C_+}$ and has zero trace on $i\R$.
Consequently it can be extended across $i\R$ by reflection. 
Hence, $\phi_r \circ \mathcal{G}$ is smooth (away from the other boundary). 
Going back to $\phi_r$, it is the composition of a smooth function by
$\mathcal{G}^{-1}:z \mapsto (i(z-i))^{2/3}$. 
It is therefore smooth up to the boundary on $\partial \Omega_1 \setminus \{i\}$ 
and locally  of class $C^{2/3}$, 
and so $\Re(\phi)$ is of class $C^{2/3}$ as well. 
This proves the regularity claim for $\Re(\phi)$. 
The one for $\Im(\phi)$ follows as well, because we can introduce the harmonic
conjugate of $\phi_r \circ \mathcal{G}$ at the level of the reflected domain, 
so that composed with $\mathcal{G}^{-1}$, this gives the harmonic conjugate of
$\phi_r$, up to an additive constant. \par
The same considerations prove also \eqref{Eq:PhiAuxCoins}, since
$\Re(\phi(z)) = \Im(z) - 1 + 
[\phi_r \circ \mathcal{G}] \circ \mathcal{G}^{-1}$, where $\phi_r \circ \mathcal{G}$
is smooth. All the same, the singularity for $\Im(\phi)$ merely comes from the composition
of a smooth function with $\mathcal{G}^{-1}$. 
This concludes the proof of Lemma~\ref{Lem:PhiHolder}.
\end{proof}
\ \par
\begin{proof}[Proof of Lemma~\ref{Lem:WeightHolderEstimate}]
It is mainly a consequence of Lemma~\ref{Lem:PhiHolder}. \par
First, we show that $\Im(\phi)$ is constant on $\mathcal{M}_1$.
Here $\mathcal{M}_1$ denotes the middle line $\mathcal{M}$ corresponding to $L=1$ 
(see Paragraph~\ref{Subsubsec:Domains}).
From \eqref{Eq:Sys-Phi}, we see that  $\tilde{\varphi}$ is symmetric 
with respect to the $y$-axis. 
With \eqref{Eq:PropPhi}, we see that $\Re(\phi)$ shares the same symmetry, 
so $\partial_x \Re(\phi) = 0$ on $\mathcal{M}_1$.
By the Cauchy-Riemann equation, $\partial_y \Im(\phi) = 0$ on $\mathcal{M}_1$,
so that $\Im(\phi)$ is constant on $\mathcal{M}_1$. \par
Using the fact that $\Im(\phi)$ is constant on $\mathcal{M}_1$ and \eqref{Eq:ProdPoidsMult},
we see that
\begin{equation*}
\left| e^\lambda \overline{\psi}^{-1} \right|_{C^{2/3}(\mathcal{M})} 
= \left| \exp\left( \frac{L^2}{4T}
+ \frac{L^2}{2T} \, \Re{\phi}\left( \frac{\cdot}{L} \right)\right) \right|_{C^{2/3}(\mathcal{M})} .
\end{equation*}
Now using the natural factorization of 
$\Re(\phi)^k(x)-\Re(\phi)^k(y)$, it is elementary to see that for any $k \in \N$,
\begin{equation*}
|\Re(\phi)^k|_{C^{2/3}(\mathcal{M})} \leq k \|\Re(\phi)\|^{k-1}_{L^{\infty}(\mathcal{M})} |\Re(\phi)|_{C^{2/3}(\mathcal{M})},
\end{equation*}
so that by summation of the power series
\begin{multline*}
\left|\exp\left(\frac{L^2}{2T}\Re(\phi)\right)\right|_{C^{2/3}(\mathcal{M})} 
\leq \exp\left(\frac{L^2}{2T}\|\Re(\phi)\|_{L^{\infty}(\mathcal{M})}\right) 
   \frac{L^2}{2T} |\Re(\phi)|_{C^{2/3}(\mathcal{M})} \\
= \exp\left(\frac{L^2}{2T}\tilde{\varphi}(0)\right) 
   \frac{L^2}{2T} |\Re(\phi)|_{C^{2/3}(\mathcal{M})}.
\end{multline*}
Now scaling by $z \mapsto z/L$, multiplying by $\exp(L^{2}/4T)$ and
using \eqref{Eq:ApparitionKappaStar}
we easily arrive at \eqref{Eq:WeightHolderEstimate}. \par
\end{proof}
\noindent
{\bf Acknowledgment.} The author warmly thanks Sylvain Ervedoza, Armand Koenig, Pierre Lissy and Marius Tucsnak for useful discussions about a preliminary version of this paper. \par
\ \\
\noindent
{\bf AI-assistance statement.} The main proofs and the writing of this paper are the author's. However the author acknowledges the use of AI-assistance (Claude, Anthropic) to:
\begin{itemize}
\item[--] check the proofs. In particular, Claude found several errors in a first version of this paper, some of which were significant. It gave useful suggestions to correct them.
\item[--] look for typos and grammatical errors.
\item[--] find the appropriate bibliography. The references were then checked by the author.
\end{itemize}
%
%
%
%
%
%
%
%
%
%
%
\bibliographystyle{plain} 
\bibliography{heat}

@article{MRR-2016,
	author = {Martin, Philippe and Rosier, Lionel and Rouchon, Pierre},
	doi = {10.1093/amrx/abv013},
	eprint = {https://academic.oup.com/amrx/article-pdf/2016/2/181/6659444/abv013.pdf},
	issn = {1687-1200},
	journal = {Applied Mathematics Research eXpress},
	month = {09},
	number = {2},
	pages = {181-216},
	title = {On the Reachable States for the Boundary Control of the Heat Equation},
	url = {https://doi.org/10.1093/amrx/abv013},
	volume = {2016},
	year = {2016}}

@article{Ervedoza-Zuazua-2011,
	author = {Ervedoza, Sylvain and Zuazua, Enrique},
	date = {2011/12/01},
	doi = {10.1007/s00205-011-0445-8},
	id = {Ervedoza2011},
	isbn = {1432-0673},
	journal = {Archive for Rational Mechanics and Analysis},
	number = {3},
	pages = {975--1017},
	title = {Sharp Observability Estimates for Heat Equations},
	url = {https://doi.org/10.1007/s00205-011-0445-8},
	volume = {202},
	year = {2011}}

@article{Miller-Cost-2006,
	author = {Miller, Luc},
	doi = {10.1137/S0363012904440654},
	eprint = {https://doi.org/10.1137/S0363012904440654},
	journal = {SIAM Journal on Control and Optimization},
	number = {2},
	pages = {762-772},
	title = {The Control Transmutation Method and the Cost of Fast Controls},
	url = {https://doi.org/10.1137/S0363012904440654},
	volume = {45},
	year = {2006}}

@book{tartar2007introduction,
	author = {Tartar, Luc},
	doi = {10.1007/978-3-540-71483-5},
	fseries = {Lecture Notes of the Unione Matematica Italiana},
	isbn = {978-3-540-71482-8},
	issn = {1862-9113},
	language = {English},
	publisher = {Berlin: Springer},
	series = {Lect. Notes Unione Mat. Ital.},
	title = {An introduction to {Sobolev} spaces and interpolation spaces},
	volume = {3},
	year = {2007},
	zbl = {1126.46001},
	zbmath = {5152732}}

@book{lions2012non,
	author = {Lions, Jacques-Louis and Magenes, Enrico},
	publisher = {Springer Science \& Business Media},
	title = {Non-homogeneous boundary value problems and applications: Vol. 1},
	volume = {1},
	year = {2012}}

@article{seidman1979time,
	author = {Seidman, Thomas I},
	journal = {Journal of Mathematical Analysis and Applications},
	number = {1},
	pages = {17--20},
	publisher = {Elsevier},
	title = {Time-invariance of the reachable set for linear control problems},
	volume = {72},
	year = {1979}}

@article{fattorini1978reachable,
	author = {Fattorini, Hector O.},
	journal = {Proceedings of the Royal Society of Edinburgh Section A: Mathematics},
	number = {1-2},
	pages = {71--77},
	publisher = {Royal Society of Edinburgh Scotland Foundation},
	title = {Reachable states in boundary control of the heat equation are independent of time},
	volume = {81},
	year = {1978}}

@unpublished{lissy2026optimal,
	author = {Pierre Lissy},
	journal = {preprint arXiv: 2608.08041},
	title = {Optimal cost of fast boundary controls for the one-dimensional heat equation},
	year = {2026}}

@book{MR507701,
	author = {Bergman, Stefan},
	edition = {revised},
	mrclass = {30A31 (30A30 32H10)},
	mrnumber = {507701},
	mrreviewer = {M.\ Skwarczy\'nski},
	pages = {x+257},
	publisher = {American Mathematical Society, Providence, RI},
	series = {Mathematical Surveys},
	title = {The kernel function and conformal mapping},
	volume = {No. V},
	year = {1970}}

@book{MR2768550,
	author = {Bahouri, Hajer and Chemin, Jean-Yves and Danchin, Rapha\"el},
	doi = {10.1007/978-3-642-16830-7},
	isbn = {978-3-642-16829-1},
	mrclass = {35-02 (35L72 35Q30 42-02 42B37 76B03 76D03 76N10)},
	mrnumber = {2768550},
	mrreviewer = {Peter\ R.\ Massopust},
	pages = {xvi+523},
	publisher = {Springer, Heidelberg},
	series = {Grundlehren der mathematischen Wissenschaften},
	title = {Fourier analysis and nonlinear partial differential equations},
	url = {https://doi.org/10.1007/978-3-642-16830-7},
	volume = {343},
	year = {2011}}

@unpublished{ortner2012note,
	author = {Ortner, Christoph and S{\"u}li, Endre},
	journal = {arXiv preprint arXiv:1202.3970},
	title = {A Note on Linear Elliptic Systems on {$\mathbb{R}^{d}$}},
	year = {2012}}

@article{Darde-Ervedoza-2018,
	author = {Dard\'e, J\'er\'emi and Ervedoza, Sylvain},
	doi = {10.1137/16M1093215},
	eprint = {https://doi.org/10.1137/16M1093215},
	journal = {SIAM Journal on Control and Optimization},
	number = {3},
	pages = {1692-1715},
	title = {On the Reachable Set for the One-Dimensional Heat Equation},
	url = {https://doi.org/10.1137/16M1093215},
	volume = {56},
	year = {2018}}

@incollection{SohrSpecovius1998,
	author = {Hermann Sohr and Maria Specovius-Neugebauer},
	booktitle = {Theory of the Navier--Stokes Equation},
	editor = {J. G. Heywood and K. Masuda and R. Rautmann and V. A. Solonnikov},
	publisher = {World Scientific},
	series = {Series on Advances in Mathematics for Applied Sciences},
	title = {The Stokes Problem on Exterior Domains in Homogeneous Sobolev Spaces},
	volume = {47},
	year = {1998}}

@article{Fattorini-Russell,
	author = {Fattorini, Hector O. and Russell, David L.},
	date = {1971/01/01},
	doi = {10.1007/BF00250466},
	id = {Fattorini1971},
	isbn = {1432-0673},
	journal = {Archive for Rational Mechanics and Analysis},
	number = {4},
	pages = {272--292},
	title = {Exact controllability theorems for linear parabolic equations in one space dimension},
	url = {https://doi.org/10.1007/BF00250466},
	volume = {43},
	year = {1971}}

@book{grisvard1985elliptic,
	author = {Grisvard, Pierre},
	isbn = {9781611972030},
	publisher = {Society for Industrial and Applied Mathematics},
	series = {Classics in Applied Mathematics},
	title = {Elliptic Problems in Nonsmooth Domains},
	url = {https://books.google.com/books?id=LdMk297ExhkC},
	year = {1985}}

@article{hartmann2020reachable,
	author = {Hartmann, Andreas and Kellay, Karim and Tucsnak, Marius},
	journal = {Journal of the European Mathematical Society},
	number = {10},
	pages = {3417--3440},
	title = {From the reachable space of the heat equation to Hilbert spaces of holomorphic functions},
	volume = {22},
	year = {2020}}

@book{BoyerFabrie2012,
	author = {Boyer, Franck and Fabrie, Pierre},
	publisher = {Springer},
	series = {Applied Mathematical Sciences},
	title = {{Mathematical Tools for the Study of the Incompressible Navier-Stokes Equations and Related Models}},
	volume = {183},
	year = {2012}}

@book{Bell2015,
	address = {Andover, England, UK},
	author = {Bell, Steven R.},
	doi = {10.1201/b19222},
	isbn = {978-0-42916289-3},
	journal = {Taylor {\&} Francis},
	month = nov,
	publisher = {Taylor {\&} Francis},
	title = {{The Cauchy Transform, Potential Theory and Conformal Mapping}},
	year = {2015}}

@article{carleman1927,
	author = {Carleman, Torsten},
	journal = {Ark. Fys.},
	number = {4},
	pages = {1--5},
	title = {{Sur un th{\'e}or{\`e}me de Weierstrass}},
	volume = {20},
	year = {1927}}

@article{Kellay-Normand-Tucsnak,
	author = {Kellay, Karim and Normand, Thomas and Tucsnak, Marius},
	doi = {10.2140/apde.2022.15.891},
	fjournal = {Analysis \& PDE},
	issn = {2157-5045,1948-206X},
	journal = {Anal. PDE},
	mrclass = {93B03 (30H20 35K08 93B05 93C20)},
	mrnumber = {4478293},
	mrreviewer = {Michela\ Egidi},
	number = {4},
	pages = {891--920},
	title = {Sharp reachability results for the heat equation in one space dimension},
	url = {https://doi.org/10.2140/apde.2022.15.891},
	volume = {15},
	year = {2022}}

@article{Hartmann-Orsoni,
	author = {Hartmann, Andreas and Orsoni, Marcu-Antone},
	doi = {10.1016/j.matpur.2021.04.009},
	fjournal = {Journal de Math\'ematiques Pures et Appliqu\'ees. Neuvi\`eme S\'erie},
	issn = {0021-7824,1776-3371},
	journal = {J. Math. Pures Appl. (9)},
	mrclass = {30H20 (35K05 46E15 93B03)},
	mrnumber = {4248466},
	mrreviewer = {Serhii\ V.\ Gryshchuk},
	pages = {181--201},
	title = {Separation of singularities for the {B}ergman space and application to control theory},
	url = {https://doi.org/10.1016/j.matpur.2021.04.009},
	volume = {150},
	year = {2021}}

@article{Orsoni,
	author = {Orsoni, Marcu-Antone},
	doi = {10.1016/j.jfa.2020.108852},
	fjournal = {Journal of Functional Analysis},
	issn = {0022-1236,1096-0783},
	journal = {J. Funct. Anal.},
	mrclass = {93B03 (30H20 35K05 44A10)},
	mrnumber = {4211029},
	mrreviewer = {Nino\ Manjavidze},
	number = {7},
	pages = {Paper No. 108852, 17},
	title = {Reachable states and holomorphic function spaces for the 1-D heat equation},
	url = {https://doi.org/10.1016/j.jfa.2020.108852},
	volume = {280},
	year = {2021}}

@article{Darde-Ervedoza-2019,
	author = {Dard\'e, J\'er\'emi and Ervedoza, Sylvain},
	doi = {10.2140/apde.2019.12.1455},
	fjournal = {Analysis \& PDE},
	issn = {2157-5045,1948-206X},
	journal = {Anal. PDE},
	mrclass = {93B05 (35K05 42A38 93C20)},
	mrnumber = {3921310},
	mrreviewer = {Emmanuel\ Tr\'elat},
	number = {6},
	pages = {1455--1488},
	title = {On the cost of observability in small times for the one-dimensional heat equation},
	url = {https://doi.org/10.2140/apde.2019.12.1455},
	volume = {12},
	year = {2019}}

@article{Lissy-2015,
	author = {Lissy, Pierre},
	doi = {10.1016/j.jde.2015.06.031},
	fjournal = {Journal of Differential Equations},
	issn = {0022-0396,1090-2732},
	journal = {J. Differential Equations},
	mrclass = {93B05 (35K20 93C20)},
	mrnumber = {3377528},
	mrreviewer = {Luz\ de Teresa},
	number = {10},
	pages = {5331--5352},
	title = {Explicit lower bounds for the cost of fast controls for some 1-{D} parabolic or dispersive equations, and a new lower bound concerning the uniform controllability of the 1-{D} transport-diffusion equation},
	url = {https://doi.org/10.1016/j.jde.2015.06.031},
	volume = {259},
	year = {2015}}

@article{Miller04,
	author = {Miller, Luc},
	doi = {10.1016/j.jde.2004.05.007},
	fjournal = {Journal of Differential Equations},
	issn = {0022-0396,1090-2732},
	journal = {J. Differential Equations},
	mrclass = {93B05 (35A22 35K05 93B07 93C20)},
	mrnumber = {2076164},
	mrreviewer = {Enrique\ Zuazua},
	number = {1},
	pages = {202--226},
	title = {Geometric bounds on the growth rate of null-controllability cost for the heat equation in small time},
	url = {https://doi.org/10.1016/j.jde.2004.05.007},
	volume = {204},
	year = {2004}}

@article{Tenenbaum-Tucsnak,
	author = {Tenenbaum, G{\'e}rald and Tucsnak, Marius},
	doi = {10.1016/j.jde.2007.06.019},
	fjournal = {Journal of Differential Equations},
	issn = {0022-0396,1090-2732},
	journal = {J. Differential Equations},
	mrclass = {93C25 (11N36 93B05 93B07 93C20)},
	mrnumber = {2363470},
	mrreviewer = {Luz\ de Teresa},
	number = {1},
	pages = {70--100},
	title = {New blow-up rates for fast controls of {S}chr\"odinger and heat equations},
	url = {https://doi.org/10.1016/j.jde.2007.06.019},
	volume = {243},
	year = {2007}}

@article{Guichal-1985,
	author = {G\"uichal, Edgardo N.},
	doi = {10.1016/0022-247X(85)90313-0},
	fjournal = {Journal of Mathematical Analysis and Applications},
	issn = {0022-247X},
	journal = {J. Math. Anal. Appl.},
	mrclass = {93C20 (49A22)},
	mrnumber = {805273},
	mrreviewer = {T.\ I.\ Seidman},
	number = {2},
	pages = {519--527},
	title = {A lower bound of the norm of the control operator for the heat equation},
	url = {https://doi.org/10.1016/0022-247X(85)90313-0},
	volume = {110},
	year = {1985}}

@article{Seidman-1984,
	author = {Seidman, Thomas I.},
	doi = {10.1007/BF01442174},
	fjournal = {Applied Mathematics and Optimization},
	issn = {0095-4616,1432-0606},
	journal = {Appl. Math. Optim.},
	mrclass = {49E05 (49A22)},
	mrnumber = {743923},
	mrreviewer = {Claudia\ Simionescu-Badea},
	number = {2},
	pages = {145--152},
	title = {Two results on exact boundary control of parabolic equations},
	url = {https://doi.org/10.1007/BF01442174},
	volume = {11},
	year = {1984}}

@book{Saitoh-Book,
	author = {Saitoh, Saburou},
	isbn = {0-582-31758-4},
	mrclass = {46E20 (44A05 46N20 47B38)},
	mrnumber = {1478165},
	mrreviewer = {Pa\c sc\ G\u avru\c ta},
	pages = {xiv+280},
	publisher = {Longman, Harlow},
	series = {Pitman Research Notes in Mathematics Series},
	title = {Integral transforms, reproducing kernels and their applications},
	volume = {369},
	year = {1997}}

@article{Aikawa-Hayashi-Saitoh,
	author = {Aikawa, Hiroaki and Hayashi, Nakao and Saitoh, Saburou},
	doi = {10.1080/17476939008814430},
	fjournal = {Complex Variables. Theory and Application. An International Journal},
	issn = {0278-1077,1563-5066},
	journal = {Complex Variables Theory Appl.},
	mrclass = {30C40 (35K05 46E22)},
	mrnumber = {1055936},
	mrreviewer = {Manfred\ W.\ Kracht},
	number = {1},
	pages = {27--36},
	title = {The {B}ergman space on a sector and the heat equation},
	url = {https://doi.org/10.1080/17476939008814430},
	volume = {15},
	year = {1990}}
\addcontentsline{toc}{section}{References}
\end{document}